\documentclass[11pt]{article}
\usepackage[margin=1in]{geometry} 
\usepackage[utf8x]{inputenc}
\usepackage[OT4]{fontenc}
\usepackage[T1]{fontenc}
\usepackage{amssymb,amsfonts,amsmath,bbm,mathrsfs,stmaryrd}
\usepackage{xcolor}
\usepackage{url}
\usepackage{relsize}
\usepackage{enumerate}
\usepackage{tikz-cd}
\usepackage[all,cmtip]{xy}
\usepackage{caption}

\usepackage[colorlinks,
             linkcolor=black!75!red,
             citecolor=blue,
             pdftitle={},
             pdfauthor={x},
             pdfproducer={pdfLaTeX},
             pdfpagemode=FullScreen,
             bookmarksopen=true,
             bookmarksnumbered=true]{hyperref}

\usepackage{tikz}
\usetikzlibrary{arrows,calc,decorations.pathreplacing,decorations.markings,intersections,shapes.geometric,through,fit,shapes.symbols,positioning,decorations.pathmorphing}

\usepackage[amsmath,thmmarks]{ntheorem}
\usepackage{cleveref}
\usepackage{comment}

\creflabelformat{enumi}{label*=\arabic*}

\crefname{section}{Section}{Sections}
\crefformat{section}{#2Section~#1#3} 
\Crefformat{section}{#2Section~#1#3} 

\crefname{subsection}{\S}{\S\S}
\crefformat{subsection}{#2\S#1#3} 
\Crefformat{subsection}{#2\S#1#3} 

\theoremstyle{plain}

\newtheorem{lemma}{Lemma}[section]
\newtheorem{proposition}[lemma]{Proposition}

\newtheorem{theorem}[lemma]{Theorem}

\theoremstyle{nonumberplain}

\theoremstyle{plain}
\theorembodyfont{\upshape}
\theoremsymbol{\ensuremath{\blacklozenge}}

\newtheorem{definition}[lemma]{Definition}
\newtheorem{example}[lemma]{Example}
\newtheorem{remark}[lemma]{Remark}

\crefname{definition}{definition}{definitions}
\crefformat{definition}{#2definition~#1#3} 
\Crefformat{definition}{#2Definition~#1#3} 

\crefname{ex}{example}{examples}
\crefformat{example}{#2example~#1#3} 
\Crefformat{example}{#2Example~#1#3} 

\crefname{remark}{remark}{remarks}
\crefformat{remark}{#2remark~#1#3} 
\Crefformat{remark}{#2Remark~#1#3} 

\crefname{convention}{convention}{conventions}
\crefformat{convention}{#2convention~#1#3} 
\Crefformat{convention}{#2Convention~#1#3} 

\crefname{claim}{claim}{claims}
\crefformat{claim}{#2claim~#1#3} 
\Crefformat{claim}{#2Claim~#1#3} 

\crefname{conjecture}{conjecture}{conjectures}
\crefformat{conjecture}{#2conjecture~#1#3} 
\Crefformat{conjecture}{#2Conjecture~#1#3}

\crefname{lemma}{lemma}{lemmas}
\crefformat{lemma}{#2lemma~#1#3} 
\Crefformat{lemma}{#2Lemma~#1#3} 

\crefname{proposition}{proposition}{propositions}
\crefformat{proposition}{#2proposition~#1#3} 
\Crefformat{proposition}{#2Proposition~#1#3} 

\crefname{question}{question}{questions}
\crefformat{question}{#2question~#1#3} 
\Crefformat{question}{#2Question~#1#3}

\crefname{corollary}{corollary}{corollaries}
\crefformat{corollary}{#2corollary~#1#3} 
\Crefformat{corollary}{#2Corollary~#1#3} 

\crefname{theorem}{theorem}{theorems}
\crefformat{theorem}{#2theorem~#1#3} 
\Crefformat{theorem}{#2Theorem~#1#3} 

\crefname{assumption}{assumption}{Assumptions}
\crefformat{assumption}{#2assumption~#1#3} 
\Crefformat{assumption}{#2Assumption~#1#3} 

\crefname{equation}{}{}
\crefformat{equation}{(#2#1#3)} 
\Crefformat{equation}{(#2#1#3)}

\theoremstyle{nonumberplain}
\theoremsymbol{\ensuremath{\blacksquare}}

\newtheorem{proof}{Proof}

\def\polhk#1{\setbox0=\hbox{#1}{\ooalign{\hidewidth
    \lower1.5ex\hbox{`}\hidewidth\crcr\unhbox0}}}

\newcommand{\bes}{\begin{equation*}}
\newcommand{\ees}{\end{equation*}}
\newcommand{\be}{\begin{equation}}
\newcommand{\ee}{\end{equation}}

\allowdisplaybreaks

\usetikzlibrary{cd}

\begin{document}

\baselineskip=15pt

\title{Gluing topological graph C*-algebras}
\author{
  Atul Gothe, John Quigg, and Mariusz Tobolski
}

\date{}

\newcommand{\Addresses}{{
  \bigskip
  \footnotesize  
  
  \textsc{Instytut Matematyczny Polskiej Akademii Nauk, ul. \'Sniadeckich 8, 00-656 Warszawa, Poland}\par\nopagebreak
  \textsc{Instytut Matematyki, Uniwersytet Warszawski, ul. Stefana Banacha 2, 02-097 Warszawa, Poland}\par\nopagebreak
  \textit{E-mail adress}:
  \texttt{a.gothe@uw.edu.pl}
  
  \bigskip

  \textsc{School of Mathematical and Statistical Sciences, Arizona State University, Tempe, AZ 85284}\par\nopagebreak \textit{E-mail address}:
  \texttt{quigg@asu.edu}

\bigskip

  \textsc{Instytut Matematyczny, Uniwersytet Wroc\l{}awski, pl. Grunwaldzki 2/4, 50-384 Wroc\l{}aw, Poland}\par\nopagebreak \textit{E-mail address}:
  \texttt{mariusz.tobolski@math.uni.wroc.pl}
}}

\maketitle

\begin{abstract}
We introduce regular closed subgraphs of Katsura's topological graphs and use them to generalize the notion of an adjunction space from topology. Our construction attaches a~topological graph onto another via a regular factor map. We prove that under suitable assumptions the C*-algebra of the adjunction graph is a pullback of the C*-algebras of the topological graphs being glued. Our results generalize certain pushout-to-pullback theorems proved in the context of discrete directed graphs. Our theorem applied to homeomorphism C*-algebras recovers a special case of the well-known result stating that pullbacks of $\mathbb{Z}$-C*-algebras induce pullbacks of the respective crossed product C*-algebras. 
Furthermore, we show that the C*-algebras of odd-dimensional quantum balls of Hong and Szyma\'nski (which are known not to be graph C*-algebras) are topological graph C*-algebras and we recover the pullback structure of C*-algebras of odd-dimensional quantum spheres by gluing the topological graphs associated to the C*-algebras of the corresponding odd-dimensional quantum balls.
\end{abstract}


\vspace{.5cm}


\section{Introduction}

The idea of understanding complex objects by analyzing their simpler components is pervasive in mathematics. In topology and differential geometry, many interesting spaces  are constructed by gluing together well-understood spaces.  For example, CW-complexes are formed by attaching topological balls of various dimensions in a controlled manner. Such constructions are considered tractable because many of the resulting spaces' features can be derived from the corresponding properties of their constituent parts. In particular, the homology and K-theory groups of adjunction spaces (or one-injective pushouts) can be obtained from those of the spaces being glued, using tools such as the Mayer--Vietoris sequence. In the theory of C*-algebras, often called noncommutative topology, the aforementioned gluing constructions translate to pullback diagrams via the Gelfand--Naimark duality. Remarkably, both the K-theory and  the Mayer--Vietoris six-term exact sequence remain applicable in this setting. 

In recent years, the following problem in the field of graph C*-algebras gained attention: Let $E$ be a directed graph obtained by a gluing construction. Under which conditions is the associated C*-algebra $C^*(E)$ a~pullback C*-algebra? We call results of this kind {\em pushout-to-pullback theorems}. The question was motivated by examples due to Hong and Szymański~\cite{hs-02} unraveling the quantum CW-complex structure (in the sense of~\cite{dhmsz-20}) of C*-algebras of even-dimensional quantum spheres. Conditions permitting the pullback structure of C*-algebras of unions of graphs without breaking vertices were given by Hajac, Reznikhoff, and the third author in~\cite{hrt-20}. Brooker and Spielberg proved the result in a~more general setting of unions of relative graphs (incorporating breaking vertices) in~\cite{bs-24}. Then, Hajac and the third author proved a general pushout-to-pullback theorem going beyond unions of graphs~\cite{th-24}. Results of the same kind were considered also for higher-rank graph C*-algebras~\cite{kpsw-16} and Cuntz--Pimsner algebras~\cite{rs-11}.

In this paper, we extend the above framework to the setting of topological graphs in the sense of Katsura~\cite{katsura1}, which generalize directed graphs by permitting arbitrary locally compact Hausdorff topologies on their vertex and edge spaces. We introduce the notion of {\em adjunction topological graphs}, inspired by the topological concept of adjunction spaces, to serve as analogues of pushouts in the category of graphs. Our construction involves attaching one topological graph to another along a~closed subgraph via a regular factor map, as defined by Katsura~\cite{katsura2}. The main theorems establish sufficient conditions under which this gluing process induces a pullback diagram at the level of associated topological graph C*-algebras. Our results generalize some existing theorems concerning graph C*-algebras found in the literature.

Beyond the general interest in extending known results, this project is motivated by developments in noncommutative topology and the theory of $q$-deformed spaces. In the case of even-dimensional quantum spheres, prior work has shown that gluing constructions for graphs correspond strikingly well with the quantum CW-complex structures of their C*-algebras. Motivated by this insight, our aim was to establish an analogous correspondence for odd-dimensional quantum spheres. A central challenge in this context is that the fundamental building blocks, namely, the C*-algebras of odd-dimensional quantum balls, cannot be realized as graph C*-algebras. In this paper, we show that these C*-algebras are in fact isomorphic to topological graph C*-algebras arising from explicit topological graphs constructed via a suspension procedure. This result enables a visualization of the quantum CW-complex structure of odd-dimensional quantum spheres through the lens of topological graphs, which forms the principal application of our work.

The introduction of non-discrete topologies naturally gives rise to more intricate gluing constructions. These, in turn, lead to a pushout-to-pullback theorem that may offer greater utility than in the purely discrete setting. For example, in the case of directed graphs, the K-theory of graph C*-algebras can be effectively computed using the Raeburn--Szymański theorem~\cite{rsz-04} (see also~\cite{kprr-97}), which provides explicit formulas in terms of the graph's adjacency matrix. This approach is so efficient that there is typically no need to invoke the Mayer--Vietoris six-term exact sequence arising from a known pullback structure. In contrast, for topological graphs, the situation is more subtle. The K-theory of a topological graph C*-algebra $C^*(E)$ can still be computed via a Pimsner--Voiculescu six-term exact sequence, which involves the K-theory of commutative C*-algebras associated to the vertex space and its subspace of regular vertices. However, this method lacks the algorithmic simplicity of the adjacency matrix approach. Consequently, alternative computational tools might become desirable in certain cases. In particular, a detailed understanding of the gluing structure of a topological graph may yield new insights and computational advantages. 

In the preliminaries, we state the basic definitions and facts regarding the Cuntz--Pimsner algebras and topological graphs needed in the sequel. In Section~3, we define regular closed subgraphs and then use them to introduce the notion of an adjuction topological graph enabling us to glue topological graphs. Then, we study various regularity conditions one might require from such gluings. Section~4 contains the pushout-to-pullback theorem for topological graphs and a discussion on how it generalizes certain results for directed graphs. Finally, in Section~5, we exemplify the main result by studying two classes of examples: homeomorphism C*-algebras and odd-dimensional quantum spheres.


\section{Preliminaries}

A {\em (strongly continuous) action} of the circle group $\mathbb{T}$ on a C*-algebra $A$ is a group ho\-mo\-mor\-phism $\alpha:\mathbb{T}\to {\rm Aut}(A)$ such that the map $\mathbb{T}\times A\ni (g,a)\mapsto\alpha_g(a):=\alpha(g)(a)\in A$ is continuous. A~C*-algebra equipped with a $\mathbb{T}$-action is called a {\em $\mathbb{T}$-C*-algebra}. Now, let $A$ and $B$ be $\mathbb{T}$-C*-algebras equipped with actions $\alpha$ and $\beta$, respectively. A $*$-homomorphism $\pi:A\to B$ is called {\em $\mathbb{T}$-equivariant} if $\pi\circ \alpha_g=\beta_g\circ \pi$ for all $g\in G$.

\subsection{Cuntz--Pimsner algebras}

General theory of Hilbert modules and C*-correspondences 
can be found in~\cite{enchilada}. The pioneering idea of 
associating C*-algebras to C*-correspondences is due to 
Pimsner~\cite{pimsner}. Here we follow~\cite{katsura0}.

A {\em C*-correspondence} over a C*-algebra $A$ is a 
(right) Hilbert $A$-module $X$ together with a~left action 
implemented by a $*$-homomorphism 
$\varphi_X:A\to \mathcal{L}(X)$, where $\mathcal{L}(X)$ is 
the C*-algebra of adjointable operators on $X$. We will often write 
$a\cdot \xi$ to denote $\varphi_X(a)(\xi)$, where $a\in A$ 
and $\xi\in X$. Next, consider the set $\mathcal{K}(X):=\overline{\rm span}\{\theta_{\xi,\eta}~:~\xi,\eta\in X\}$, where
\[
\theta_{\xi,\eta}(\zeta):=\xi\cdot\langle\eta,\zeta\rangle,\qquad \zeta\in X.
\]
Then $\mathcal{K}(X)$ is a closed ideal of $\mathcal{L}(X)$ and elements of $\mathcal{K}(X)$ are called {\em compact operators on $X$}. 
\begin{definition}\label{correp}
(\cite[Definition~2.1]{katsura0})
A~{\em representation} of a C*-correspondence $X$ over $A$ on a C*-algebra $B$ is a pair $(\psi^1,\psi^0)$ consisting of a linear map $\psi^1:X\to B$ and a $*$-homomorphism $\psi^0:A\to B$ satisfying
\begin{enumerate}
\item[(R1)] $\psi^1(a\cdot \xi)=\psi^0(a)\psi^1(\xi)$ for all $a\in A$ and $\xi\in X$.
\item[(R2)] $\psi^0(\langle\xi,\eta\rangle)=\psi^1(\xi)^*\psi^1(\eta)$ for all $\xi,\eta\in X$.
\end{enumerate} 
\end{definition}
For a C*-correspondence $X$ over $A$, the Katsura ideal of $A$ is given by
\[
J_X:=\{a\in A~:~\varphi_X(a)\in\mathcal{K}(X)\text{ and }ab=0\text{ for all $b\in\ker\varphi$}\}.
\]
Next, for every representation $(\psi^1,\psi^0)$ of $X$ on $B$, we define a $*$-homomorphism $\psi^{(1)}:\mathcal{K}(X)\to B$ by the formula 
\[
\psi^{(1)}(\theta_{\xi,\eta}):=\psi^1(\xi)\psi^1(\eta)^*
\]
for all $\xi,\eta\in X$.
\begin{definition}(\cite[Definition~3.4]{katsura0})
A representation $(\psi^1,\psi^0)$ of a C*-correspondence $X$ over $A$ on a C*-algebra $B$ is {\em covariant} if $\psi^0(a)=\psi^{(1)}(\varphi_X(a))$ for all $a\in J_X$. 
\end{definition}

\begin{definition}(\cite[Definition~3.5]{katsura0}, cf.~\cite{pimsner})
The {\em Cuntz--Pimsner algebra} $\mathcal{O}_X$ associated to a correspondence $X$ over $A$ is the C*-algebra $C^*(t^1_X,t^0_X)$ generated by the universal covariant representation $(t^1_X,t^0_X)$ of $(X,A)$. 
\end{definition}
By universality, for every covariant representation $(\psi^1,\psi^0)$ of $X$ on $B$, there is a~unique $*$-ho\-mo\-mor\-phism $\psi:\mathcal{O}_X\to B$ onto the C*-algebra $C^*(\psi^1,\psi^0)$ (generated by the images of $\psi^1$ and $\psi^0$) such that $\psi\circ t^1_X=\psi^1$ and $\psi\circ t^0_X=\psi^0$. Next, the formulas
\begin{equation}\label{cuntzgauge}
\alpha_z(t^0_X(a)):=t^0_X(a),\qquad \alpha_z(t^1_X(\xi)):=zt^1_X(\xi),\qquad z\in\mathbb{T},\quad a\in A,\quad \xi\in X,
\end{equation}
give rise to an action $\alpha:\mathbb{T}\to{\rm Aut}(\mathcal{O}_X)$, called the~{\em gauge action}. A  $*$-homomorphism between Cuntz--Pimsner algebras which is $\mathbb{T}$-equivariant with respect to gauge actions is called {\em gauge-equivariant}. If $I$ is an ideal of $\mathcal{O}_X$ such that $\alpha_z(I)\subseteq I$ for all $z\in\mathbb{T}$, then $I$ is {\em gauge-invariant}.

\subsection{Topological graphs}

We refer the reader to~\cite{katsura1,katsura2,katsura3} for an extensive study of topological graphs and their C*-algebras (note, however, that we use a different source-range convention). A {\em topological graph} $E$ is a~quadruple $(E^0,E^1,s_E,r_E)$, where $E^0$ and $E^1$ are locally compact Hausdorff spaces and $s_E,r_E:E^1\to E^0$ are continuous maps such that $r_E$ is a local homeomorphism. We call $E^0$ the {\em vertex space}, $E^1$ the {\em edge space}, and $s_E$ and $r_E$ the {\em source} and the {\em range} map, respectively.

Let us recall the construction of the C*-algebra of a topological graph. Let $A_E:=C_0(E^0)$ and for continuous functions $\xi,\eta\in C(E^1)$ define the following function on $E^0$:
\begin{equation}\label{inner}
\langle \xi,\eta\rangle(v):=\sum_{e\in r^{-1}_E(v)}\overline{\xi(e)}\eta(e).
\end{equation}
Next, define the following subspace of the space of bounded functions on $E^1$:
\[
X_E:=\left\{\xi\in C_b(E^1)~:~\langle\xi,\xi\rangle\in A_E\right\}.
\] 
Then $(X_E,A_E)$ is a C*-correspondence with the $A_E$-valued inner-product~\eqref{inner} and the actions
\[
(\xi\cdot f)(e):=\xi(e)f(r_E(e)),\qquad (f\cdot\xi)(e):=f(s_E(e))\xi(e),\qquad f\in A_E,\quad \xi\in X_E.
\]
The {\em C*-algebra $C^*(E)$ of a topological graph $E$} is the Cuntz--Pimsner algebra of the C*-cor\-re\-spon\-dence $X_E$ over $A_E$~(see~\cite[Section~2]{katsura1}).  The class of C*-algebras of topological graphs includes graph C*-algebras and homeomorphism C*-algebras. 
\vspace*{2mm}

\noindent{\bf Notation}. In what follows, let $\varphi_E$ denote the map defining the left action on $X_E$ and let $(t_E^1,t_E^0)$ denote the universal covariant representation.
\vspace*{2mm}

One distinguishes the following subsets of the vertex space $E^0$:
\begin{align*}
E_{\rm fin}^0&:=\{v\in E^0~:~\text{there is a neighborhood $V$ of $v$ such that $s_E^{-1}(V)$ is compact}\},\\
E^0_{\rm inf}&:=E^0\setminus E^0_{\rm fin},\qquad E_{\rm sink}^0:=E^0\setminus \overline{s_E(E^1)},\\
E_{\rm reg}^0&:=E^0_{\rm fin}\setminus \overline{E^0_{\rm sink}},\qquad E_{\rm sing}^0:=E^0\setminus E^0_{\rm reg}.
\end{align*}
The elements of $E_{\rm sink}^0$ are called {\em sinks}, the elements of $E^0_{\rm inf}$ are called {\em infinite emitters}, the elements of $E_{\rm reg}^0$ are called {\em regular}, and those of $E^0_{\rm sing}$ are called {\em singular}. Recall that a topological graph $F$ is called {\em row-finite} if $s_F(F^1)=F^0_{\rm reg}$~\cite[Definition~3.22]{katsura3}. Note that this notion generalizes row-finiteness of discrete graphs~(see, e.g.~\cite{bprs00}) but is not equivalent to the condition $F^0=F^0_{\rm fin}$ for general topological graphs (see Examples~3.21 and~3.23 in~\cite{katsura3}).

A subset $Y^0$ of $E^0$ is {\em positively invariant} if $r_E(e)\in Y^0$ implies $s_E(e)\in Y^0$ for every $e\in E^1$ and {\em negatively invariant} if for all $v\in Y^0\cap E^0_{\rm reg}$ there is $e\in E^1$ satisfying $s_E(e)=v$ and $r_E(e)\in Y^0$. We say that $Y^0$ is {\em invariant} if it is both positively and negatively invariant.
\begin{definition}
(\cite[Definition~2.3]{katsura3})
An {\em admissible pair} $\rho$ is a pair $(Y^0,Z^0)$ of closed subsets of $E^0$ such that $Y^0$ is invariant and $Y^0_{\rm sing}\subseteq Z^0\subseteq E^0_{\rm sing}\cap Y^0$.
\end{definition}

Let $I$ be a gauge-invariant ideal of $C^*(E)$. By~\cite[Proposition~2.8]{katsura3}, there is an admissible pair $\rho_I=(Y^0_I,Z^0_I)$ given by
\begin{align}\label{admideal}
Y^0_I&:=\{v\in E^0~:~f(v)=0\text{ for all }f\in C_0(E^0)\text{ such that }t^0_E(f)\in I\},\\
Z^0_I&:=\{v\in E^0~:~f(v)=0\text{ for all }f\in C_0(E^0)\text{ such that }t^0_E(f)\in I+(t^1_E)^{(1)}(\mathcal{K}(X_E))\}.\nonumber
\end{align}
Conversely, given an admissible pair $\rho=(Y^0,Z^0)$, there exists an associated gauge-invariant ideal $I_\rho$  (see~\cite[Proposition~3.5]{katsura3}). We do not need the invoke the explicit form of this ideal here. Instead, we will rely on the following theorem.

\begin{theorem}\label{idealclass}
{\rm (\cite[Theorem~3.19]{katsura3})}
For a topological graph $E$, the assignment $I\mapsto \rho_I=(Y^0_I,Z^0_I)$ and $\rho=(Y^0,Z^0)\mapsto I_\rho$ give rise to an inclusion-reversing one-to-one correspondence between the set of gauge-invariant ideals of $C^*(E)$ and the set of admissible pairs in $E^0$.
\end{theorem}

Next, we need morphisms between topological graphs. 
Recall that a continuous map $f:X\to Y$ between locally compact Hausdorff spaces is {\em proper} if $f^{-1}(K)$ is compact for every compact subset $K$ of $Y$.
\begin{definition}\label{factorkatsura}
(\cite[Definition~2.1, Definition~2.6]{katsura2})
A {\em regular factor map} $m:E\to F$ between topological graphs $E$ and $F$ is a pair of proper continuous maps $m^1:E^1\to F^1$ and $m^0:E^0\to F^0$ such that
\begin{enumerate}
\item[(F1)] $r_F(m^1(e))=m^0(r_E(e))$ and $s_F(m^1(e))=m^0(s_E(e))$ for all $e\in E^1$.
\item[(F2)] If $x\in F^1$ and $v\in E^0$ satisfies $r_F(x)=m^0(v)$, then there exists a unique element $e\in E^1$ such that $m^1(e)=x$ and $r_E(e)=v$.
\item[(F3)] $m^0(E^0_{\rm sing})\subseteq F^0_{\rm sing}$.
\end{enumerate}
\end{definition}
\begin{remark}
The definition of regularity in~\cite[Definition~2.6]{katsura2} is slightly different. However, it is equivalent to our definition due to~\cite[Lemma~2.7]{katsura2}.
\end{remark}
Given a regular factor map $m:E\to F$, one defines the maps
\begin{equation}\label{precomp}
\mu^0:A_F\to A_E,\qquad (\mu^0(f))(v):=f(m^0(v)),\qquad
\mu^1:X_F\to X_E,\qquad (\mu^1(\xi))(e):=\xi(m^1(e)).
\end{equation}
The following result justifies the definition of a regular factor map.
\begin{proposition}\label{functc*}
{\rm (\cite[Proposition~2.9]{katsura2})}
Let $m:E\to F$ be a regular factor map. Then there exists a unique gauge-equivariant $*$-homomorphism $\mu:C^*(F)\to C^*(E)$ such that $\mu\circ t_F^i=\mu^i$, $i=0,1$, where $\mu^0$ and $\mu^1$ are defined by~\eqref{precomp}.
\end{proposition}


\section{Attaching subgraphs along factor maps}

Adjunction spaces are fundamental in topology, in particular, in the theory of CW-complexes (see, e.g.~\cite{t-td08}). In this section, we introduce the notion of an adjunction graph which enables us to glue  topological graphs.

\subsection{Regular closed subgraphs}
Open subgraphs of topological graphs were defined in~\cite[Definition~5.1]{katsura2}. Here we study well-behaved closed subgraphs. 
\begin{definition}
We say that a topological graph $G$ is a {\em subgraph} of a topological graph $F$, denoted $G\subseteq F$, if there are topological embeddings $G^1\subseteq F^1$ and $G^0\subseteq F^0$ such that $s_F|_{G^1}=s_G$ and $r_F|_{G^1}=r_G$. A~subgraph $G$ of $F$ is called {\em closed} if the subspaces $G^1\subseteq F^1$ and $G^0\subseteq F^0$ are closed. Finally, a~closed subgraph $G$ of $F$ is called {\em regular} if $G^0$ is a negatively invariant subset of $F^0$ and $r_F^{-1}(G^0)\subseteq G^1$.
\end{definition}

Next, we prove several useful results regarding regular closed subgraphs.
\begin{proposition}\label{injreg}
Let $G$ be a closed subgraph of a topological graph $F$. The following are equivalent:
\begin{enumerate}
\item[{\rm (1)}] $G\subseteq F$ is  regular,
\item[{\rm (2)}] the embeddings $G^1\subseteq F^1$ and $G^0\subseteq F^0$ give rise to a regular factor map $m:G\to F$.
\end{enumerate}
\end{proposition}
\begin{proof}
Note that the associated closed embeddings $m^1$ and $m^0$ are automatically proper. Similarly, the condition (F1) for $m$ is satisfied due to the assumption that the source and range maps on $G$ are the restrictions of the source and range maps on $F$.  Let us show that $m$ satisfies (F2) if and only if $r_F^{-1}(G^0)\subseteq G^1$. First, if $r_F(e)\in G^0$ for some $e\in F^1$, then $e\in G^1$ by (F2). Conversely, assume that $r_F^{-1}(G^0)\subseteq G^1$. Since $G^1=r_G^{-1}(G^0)\subseteq r_F^{-1}(G^0)$, we obtain that $G^1=r_F^{-1}(G^0)$. Hence, if $e\in F^1$ satisfies $r_F(e)\in G^0$, then $e\in r_F^{-1}(G^0)=G^1$, which shows the condition (F2) for $m$. Notice that the condition $r_F^{-1}(G^0)\subseteq G^1$ is stronger than positive invariance of $G^0\subseteq F^0$. By~\cite[Proposition~2.2]{katsura3}, the negative invariance of $G^0$ is equivalent to the condition (F3) for $m$.
\end{proof}

\begin{proposition}\label{facregsub}
Let $m:G\to E$ be a regular factor map. Then
\[
m(G):=(m^1(G^1),m^0(G^0),s_E|_{m^1(G^1)},r_E|_{m^1(G^1)})
\]
is a regular closed subgraph of $F$.
\end{proposition}
\begin{proof}
To prove that $m(G)$ is a subgraph, note that $r_E(m(G)^1)\subseteq m(G)^0$ and $s_E(m(G)^1)\subseteq m(G)^0$ due to the condition (F1) for $m$. Since $m^0$ is a~proper map, we infer that $m^0(G^0)$ is closed in~$E^0$. Therefore, $m(G)$ is a~closed subgraph of $E$. To conclude regularity of $m(E)$, we have to prove that $r_E^{-1}(m^0(G^0))\subseteq m^1(G^1)$ and that $m^0(G^0)$ is negatively invariant. Let $e\in E^1$ be such that $r_E(e)=m^0(g)$ for some $g\in G^0$. Then there is $a\in G^1$ such that $m^1(a)=e$ by (F2). To prove negative invariance, note that if $v\in m^0(G^0)\cap E^0_{\rm reg}$ then there is $g\in G^0$ such that $m^0(g)=v$. Since $m$ is regular, it follows that $g\in G^0_{\rm reg}$. In turn, there exists $a\in G^1$ such that $s_G(a)=g$. Hence, $s_E(m^1(a))=m^0(s_G(a))=m^0(g)=v$. Furthermore, $r_E(m^1(a))\in m^0(E^0)$, so $m^0(E^0)$ is negatively invariant.
\end{proof}

We end with two immediate results for row-finite graphs.
\begin{proposition}\label{rowsub}
Every regular closed subgraph $G$ of a row-finite graph $F$ is row-finite.
\end{proposition}
\begin{proof}
We need to prove that $s_G(G^1)=G^0_{\rm reg}$. The containment $G^0_{\rm reg}\subseteq s_G(G^1)$ follows from~\cite[Proposition~2.8]{katsura1}. Now, let $v\in s_G(G^1)$. Then $v=s_G(e)=s_F(e)$ for some $e\in G^1$. Since $F$ is row-finite, $v\in F^0_{\rm reg}$. Hence, $v\in G^0_{\rm reg}$ by regularity of $G\subseteq F$, and we are done.
\end{proof}

\begin{proposition}\label{rowreg}
Let $m:G\to E$ be a regular factor map between row-finite graphs. Then $m^0(G^0_{\rm reg})\subseteq E^0_{\rm reg}$.
\end{proposition}
\begin{proof}
Take $w\in G^0_{\rm reg}$. Then $w=s_G(a)$ for some $a\in G^1$ because $G$ is row-finite. Next, $m^0(w)=m^0(s_G(a))=s_E(m^1(a))$, which implies that $m^0(w)\in E^0_{\rm reg}$ by row-finiteness of $E$.
\end{proof}

\subsection{Adjunction graphs}
Now, we turn our attention to pushouts of topological graphs. In the category of (discrete) directed graphs and graph homomorphisms, pushouts come from the corresponding pushouts of the vertex and edge spaces in the category of sets and maps (see, e.g.~\cite{ek-79}). The situation is more complex in the context of topological graphs. For starters,  pushouts of locally compact Hausdorff spaces might be ill-behaved, i.e. result in spaces which are neither locally compact nor Hausdorff. 
Therefore, we consider the following special case. Let $X$ and $Y$ be topological spaces, let $A$ be a closed subspace of $Y$, and let $f:A\to X$ be a~continuous map. The {\em adjunction space} $X\cup_f Y$ (see, e.g.~\cite[p.~7]{t-td08}) is the quotient space of the disjoint union $X\sqcup Y$, where the quotient map is given by the relation $f(a)\sim a$ for all $a\in A$. The map $f$ is often called the {\em attaching map}. Note that $X\cup_fY=X\sqcup (Y\setminus A)$ as sets. 
The canonical map $X\to X\cup_f Y$ is a~closed embedding, so we can view $X$ as a~closed subspace of $X\cup_f Y$. Similarly, one defines the map
\begin{equation}\label{canadjmap}
p:Y\hookrightarrow X\sqcup Y\longrightarrow X\cup_fY,
\end{equation}
where the second arrow is the defining quotient map. Note that $p$ restricted to $Y\setminus A$ is an open embedding.

If $X$ and $Y$ are Hausdorff, then $X\cup_f Y$ is again Hausdorff (see, e.g.~\cite[Proposition 1.2.4(3)]{t-td08}). 
The following result is probably well known but since we could not find a reference, we provide a~proof.
\begin{proposition}\label{adjprop}
Let $X$ and $Y$ be locally compact Hausdorff spaces, let $A$ be a closed subspace of~$Y$, and let $f:A\to X$ be a proper map. Then $X\cup_f Y$ is a locally compact Hausdorff space.
\end{proposition}
\begin{proof}
We already know that $X\cup_f Y$ is Hausdorff so we only need to check whether every point has a~compact neighborhood. The fact that $f$ is closed combined with the definition of the topology on $X\cup_f Y$ implies that the subspaces $X\setminus f(A)$ and $Y\setminus A$ can be identified with open subsets of locally compact Hausdorff spaces $X$ and $Y$, respectively. Hence, it suffices to find compact neighborhoods of the points $f(a)\in f(A)\subseteq X\cup_f Y$. Let $K$ be a compact neighborhood of $f(a)$ in $X$. We can choose $K$ such that $K\subseteq f(A)$ because $f(A)$ is closed in $X$. Next, properness of $f$ implies that $f^{-1}(K)$ is a compact neighborhood of $a$ in $A\subseteq Y$.  Therefore, the image of the compact set $K\sqcup f^{-1}(K)$ under the quotient map is a compact neighborhood of $f(a)$ in $X\cup_f Y$.
\end{proof}
\begin{remark}
If $f$ is not proper, then $X\cup_f Y$ might not be locally compact even when $X$ is compact Hausdorff  and $Y$ is second-countable locally compact Hausdorff as the following example\footnote{We owe this example to Alexander Ravsky.} shows: $X=\{0\}\sqcup\{1/n~:~n\in\mathbb{N}\}$, $A=\{0\}\times\mathbb{N}\subseteq X\times\mathbb{N}=Y$, and $f:A\ni (0,n)\mapsto 1/n\in X$.
\end{remark}

Let us introduce an analog of the adjunction space for topological graphs using the notions of a~closed subgraph and a~factor map. Let $E$ and $F$ be topological graphs. Next, let $G$ be a closed subgraph of $F$ and let $m^0:G^0\to E^0$ and $m^1:G^1\to E^1$ be continuous maps. Then, we can consider adjunction spaces $E^0\cup_{m^0}F^0$ and $E^1\cup_{m^1}F^1$ and their defining quotient maps
\begin{equation}\label{quotadj}
q^0:E^0\sqcup F^0\longrightarrow E^0\cup_{m^0}F^0,\qquad q^1:E^1\sqcup F^1\longrightarrow E^1\cup_{m^1}F^1.
\end{equation}
Furthermore, we can view $E^0$ and $E^1$ as closed subspaces of $E^0\cup_{m^0}F^0$ and $E^1\cup_{m^1}F^1$, respectively, and we have the canonical maps
\begin{equation}\label{canmapp}
p^0:F^0\longrightarrow E^0\cup_{m^0}F^0,\qquad p^1:F^1\longrightarrow E^1\cup_{m^1}F^1,
\end{equation}
defined analogously to~\eqref{canadjmap}. Let us also mention that the diagrams 
    \[
    \begin{tikzcd}[ampersand replacement=\&]
	\& {(E\cup_mF)^0} \\
	{E^0} \&\& {F^0} \\
	\& {G^0}
	\arrow[, right hook->, from=2-1, to=1-2]
	\arrow["m^0", <-, from=2-1, to=3-2]
	\arrow["p^0", <-, from=1-2, to=2-3]
	\arrow[, right hook->, from=3-2, to=2-3]
\end{tikzcd}\qquad  \text{and}\qquad 
\begin{tikzcd}[ampersand replacement=\&]
	\& {(E\cup_mF)^1} \\
	{E^1} \&\& {F^1} \\
	\& {G^1}
	\arrow[, right hook->, from=2-1, to=1-2]
	\arrow["m^1", <-, from=2-1, to=3-2]
	\arrow["p^1", <-, from=1-2, to=2-3]
	\arrow[, right hook->, from=3-2, to=2-3]
\end{tikzcd}
\]
are pushout diagrams in the category of topological spaces and continuous maps.

To obtain a topological graph from the above data, we need the source and range maps. In the definition below, we provide sufficient assumptions to ensure that the resulting spaces are locally compact Hausdorff and the range map is a local homeomorphism.

\begin{definition}\label{adjgraph}
Let $E$ and $F$ be topological graphs and let $G$ be a closed subgraph of $F$ such that $r_F^{-1}(G^0)\subseteq G^1$. 
Furthermore, let $m^0:G^0\to E^0$ and $m^1:G^1\to E^1$ be a~pair of proper maps satisfying (F1). The {\em adjunction graph} $E\cup_m F$ is defined as follows:
\[
(E\cup_mF)^0:=E^0\cup_{m^0}F^0,\qquad (E\cup_mF)^1:=E^1\cup_{m^1}F^1,
\]
\[
s_\cup(e):=\begin{cases}s_E(e) & e\in E^1\\ p^0(s_F(e)) & e\in F^1\setminus G^1\end{cases},\qquad\text{and}\qquad r_\cup(e):=\begin{cases}r_E(e) & e\in E^1\\ r_F(e) & e\in F^1\setminus G^1\end{cases},
\]
where $p^0:F^0\to (E\cup_{m}F)^0$ is defined by~\eqref{canmapp}. When $m$ is injective, we write $E\cup F:=E\cup_mF$ and call $E\cup F$ the {\em union graph} of $E$ and $F$. In this case, we also call $G$ the {\em intersection graph} and denote it by $E\cap F$.
\end{definition}
We emphasize that the need for $E\cap F$ to be a closed subgraph in both $E$ and $F$ is dictated by the fact that general unions might not be pushouts.
\begin{proposition}
The adjunction graph $E\cup_mF$ is a~topological graph. 
\end{proposition}
\begin{proof}
The spaces $(E\cup_mF)^0$ and $(E\cup_mF)^1$ are locally compact Hausdorff due to Proposition~\ref{adjprop}. Since $r_F^{-1}(G^0)\subseteq G^1$, we infer that $r_F(F^1\setminus G^1)\subseteq F^0\setminus G^0$. Recall that we can view $E^0$ and $F^0\setminus G^0$ as subspaces of $(E\cup_{m}F)^0$. Hence, the image of $r_\cup$ is contained in $(E\cup_{m}F)^0$. The source map is automatically well defined but one has to use $p^0$ because there are edges in $F^1$ that start in $G^0$ and do not belong to $G^1$. 

Continuity of $s_\cup$ is equivalent to the continuity of the maps
\[
E^1\hookrightarrow E^1\cup_{m^1}F^1\overset{s_\cup}{\longrightarrow} E^0\cup_{m^0}F^0\qquad\text{and}\qquad  F^1\overset{p^1}{\longrightarrow} E^1\cup_{m^1}F^1\overset{s_\cup}{\longrightarrow} E^0\cup_{m^0}F^0.
\]
The first map coincides with $s_E$ by definition. As for the second map, it coincides with $p^0\circ s_F$ due to the condition (F1) for $m$. A similar argument proves the continuity of $r_\cup$.

Let $U\subseteq E^1\cup_{m^1}F^1$ be an open subset. It is clear that
\[
r_\cup(U)=r_E((q^1)^{-1}(U)\cap E^1)\sqcup r_F((q^1)^{-1}(U)\cap (F^1\setminus G^1))\subseteq E^0\cup_{m^0}F^0,
\]
where $q^1:E^1\sqcup F^1\to E^1\cup_{m^1}F^1$ is the quotient map~\eqref{quotadj}.
Recall that $U$ is open in $E^1\cup_{m^1}F^1$ if and only if $(q^1)^{-1}(U)\cap E^1$ is open in $E^1$ and $(q^1)^{-1}(U)\cap F^1$ is open in $F^1$. Since $G^1$ is a closed subset of $F^1$ and both $r_E$ and $r_F$ are local homeomorphisms, we conclude that $r_\cup$ is also a local homeomorphism. 
\end{proof}

Finally, let us compare our Definition~\ref{adjgraph} with a similar construction coming from~\cite{katsura2,katsura3} which we will need in the next section. Let 
$\rho=(Y^0,Z^0)$ be an admissible pair in a~topological graph $E$. 
First, define
\[
Y:=\left(r_E^{-1}(Y^0),Y^0,s_E|_{r_E^{-1}(Y^0)},r_E|_{r_E^{-1}(Y^0)}\right)\qquad\text{and}\qquad W^0_{\rho}:=Z^0\cap Y^0_{\rm reg}\,.
\] 
Note that $Y$ is a~regular closed subgraph of $E$, which implies that $W_\rho^0$ is well defined. Next, let $\partial W^0_\rho$ denote the boundary $\overline{W^0_\rho}\setminus W^0_\rho$, where the closure of $W^0_\rho$ is taken in $Y^0$. Then $\partial W^0_\rho$ is closed in $\overline{W^0_\rho}$ and we can consider the closed embeddings 
$\partial^0:\partial W^0_\rho\to Y^0$ and $\partial^1:r_Y^{-1}(\partial W^0_\rho)\to Y^1$. Finally, put $\overline{W_\rho^1}:=r_Y^{-1}(\overline{W^0_\rho})$ and $\partial W^1_\rho:=r_Y^{-1}(\partial W^0_\rho)$. From the aforementioned data, one constructs a~topological graph $E_\rho$~\cite[Definition~3.13]{katsura3} (c.f.~\cite[p.~799]{katsura2}) as follows:
\begin{equation}\label{katspush}
E^0_\rho:=Y^0\cup_{\partial^0}\overline{W_\rho^0},\qquad E^1_\rho:=Y^1\cup_{\partial^1}\overline{W_\rho^1},
\end{equation}
\[
s_{E_\rho}(e):=\begin{cases}s_Y(e) & e\in Y^1\\ s_Y(e)\in Y^0 & e\in \overline{W_\rho^1}\setminus \partial W^1_\rho\end{cases}\qquad\text{and}\qquad r_{E_\rho}(e):=\begin{cases}r_Y(e) & e\in Y^1\\ r_Y(e)\in W^0_\rho & e\in\overline{W_\rho^1}\setminus \partial W^1_\rho\end{cases}.
\]
Although there are many similarities between $E_\rho$ and the adjunction graph $Y\cup_\partial\overline{W_\rho}$ (assuming the latter exists), the graph $E_\rho$ is not an adjunction graph because its source map is defined differently. In particular, $s_{E_\rho}^{-1}(v)=\emptyset$ for every $v\in W^0_\rho\subseteq E^0_\rho$. Such a construction was possible due to the fact that $W^0_\rho$ is a subset of $Y^0$. The class of graphs $E_\rho$ is used in~~\cite[Proposition~3.16 and Theorem~3.19]{katsura3} to classify gauge-equivariant ideals of topological graph C*-algebras.

\subsection{Regular adjunction graphs}

Finally, we consider attaching regular closed subgraphs along regular factor maps.
\begin{definition}\label{regadj}
An adjunction graph $E\cup_mF$ is called {\em regular} if
\begin{enumerate}
\item $G$ is a regular closed subgraph of $F$ and $E$ is a regular closed subgraph of $E\cup_mF$,
\item $m:G\to E$ and $p:F\to E\cup_mF$, where $p$ is given by~\eqref{canmapp}, are regular factor maps.
\end{enumerate}
\end{definition}

Let us now examine the following question: Given a regular closed subgraph $G$ of $F$ and a regular factor map $m:G\to E$, does it follow that $E$ is necessarily a regular subgraph of $E\cup_mF$? Moreover, is the map $p:F\to E\cup_mF$, as described in the preceding lemma, necessarily a regular factor map? Such a result would generalize~\cite[Lemma~3.9]{th-24} proved in the context of discrete directed graphs. Although most of the arguments used in the discrete case can be generalized to the topological setting, the proof breaks at some point because topological graphs may have singular vertices that are neither sinks nor infinite emitters. In what follows, we obtain various partial results and provide additional conditions under which the answer to the posed question is positive. 

\begin{proposition}\label{pprop}
Let $E\cup_mF$ be an adjunction topological graph. Then $E$ is a closed subgraph of $E\cup_mF$. Furthermore, the maps $p^0:F^0\to (E\cup_mF)^0$ and $p^1:F^1\to (E\cup_mF)^1$, defined as in~\eqref{canmapp}, are proper and satisfy the condition {\rm (F1)}.
\end{proposition}
\begin{proof}
We already know that $E^1\subseteq E^1\cup_{m^1}F^1$ and $E^0\subseteq E^0\cup_{m^0}F^0$ are closed subsets. 
By definition of $s_\cup$ and $r_\cup$, we conclude that $E$ is a closed subgraph of $E\cup_mF$.
Further, the condition (F1) for $p$ is immediate from the definitions of the maps involved. We only prove that $p^0$ is proper because the proof for $p^1$ is analogous. It suffices to show that $p^0$ is closed and that $(p^0)^{-1}(v)$ is compact for every $v\in (E\cup_mF)^0$. Since $G^0$ is closed in $F^0$ and $m^0$ is a closed map, the quotient map $q^0:E^0\sqcup F^0\to (E\cup_mF)^0$ is closed. Therefore, $p^0$ is closed as the restriction of $q^0$ to a closed subset. Next, note that
\[
(p^0)^{-1}(v)=\begin{cases}\emptyset & v\in E^0\setminus m^0(G^0)\\
(m^0)^{-1}(v) & v\in m^0(G^0)\\
\{v\} & v\in F^0\setminus G^0\end{cases}.
\]
Since $(m^0)^{-1}(v)$ is compact for every $v\in m^0(G^0)$, we conclude that $p^0$ is proper.
\end{proof}

Due to the fact that $r$ is a local homeomorphism,  the condition (F2) is the same for topological graphs and discrete directed graphs (see~\cite[Definition~2.2]{th-24}). Therefore, by repeating a part of the proof of~\cite[Lemma~3.9]{th-24}, we obtain the following result.

\begin{proposition}\label{adjtb}
Let $E\cup_mF$ be an adjunction topological graph such that $r_F^{-1}(G^0)\subseteq G^1$ and $m:G\to E$ satisfies the condition {\rm (F2)}. Then $r_\cup^{-1}(E^0)\subseteq E^1$ and $p:F\to E\cup_mF$ satisfies the condition {\rm (F2)}.
\end{proposition}

Next, we prove a result concerning regularity.
\begin{proposition}\label{adjequi}
Let $E\cup_mF$ be an adjunction graph such that $G^0$ is negetively invariant in $F^0$ and $m:G\to E$ satisfies the condition {\rm (F3)}. 
Then the following conditions are equivalent:
\begin{enumerate}
\item[{\rm (1)}] $E^0$ is negatively invariant in $(E\cup_mF)^0$, 
\item[{\rm (2)}] $p=(p^1,p^0)$, where $p^1:F^1\to (E\cup_mF)^1$ and $p^0:F^0\to (E\cup_mF)^0$ are defined by~\eqref{canmapp}, satisfies the condition {\rm (F3)}.
\end{enumerate}
\end{proposition}
\begin{proof}
(1) $\Rightarrow$ (2). Let $p^0(w)\in (E\cup_mF)^0_{\rm reg}$ for some $w\in F^0$. There exists a compact neighborhood $W\subseteq (E\cup_mF)^0_{\rm reg}$ of $p^0(w)$ such that $s_\cup^{-1}(W)$ is compact and $s_\cup(s_\cup^{-1}(W))=W$. Consider the neighborhood $(p^0)^{-1}(W)$ of $w$. Then, $s_F^{-1}((p^0)^{-1}(W))=(p^1)^{-1}(s_\cup^{-1}(W))$ because $p$ satisfies (F1). By properness of $p^1$, we get that $s_F^{-1}((p^0)^{-1}(W))$ is compact. 

We need to prove that $(p^0)^{-1}(W)\subseteq s_F(s_F^{-1}((p^0)^{-1}(W)))$. Take any $w\in F^0$ such that $p^0(w)\in W=s_\cup(s_\cup^{-1}(W))\subseteq (E\cup_mF)^0_{\rm reg}$. Then, there is $e\in (E\cup_mF)^1$ such that $s_\cup(e)=p^0(w)$. If $e\in E^1$, then $s_\cup(e)=s_E(e)=p^0(w)$. This implies that $w\in G^0$ and $m^0(w)=s_E(e)$. Since $E$ is a regular closed subgraph and $m$ satisfies (F3), we conclude that $w\in G^0_{\rm reg}$. Therefore, there is $a\in G^1$ such that $s_F(a)=s_G(a)=w$. If $e\in F^1\setminus G^1$, then $s_\cup(e)=p^0(s_F(e))=p^0(w)$. Then, either both $s_F(e)$ and $w$ belong to $ F^0\setminus G^0$ or they both belong to $G^0$. In the former case, $s_F(e)=w$ and we are done. In the later case, $m^0(w)=p^0(s_F(e))$. Again, since $E$ is a regular closed subgraph and $m$ satisfies (F3), we conclude that $w\in G^0_{\rm reg}$.

(2) $\Rightarrow$ (1). Let $v\in E^0\cap (E\cup_mF)^0_{\rm reg}$ and let $V$ be a compact neighborhood of $v$ in $(E\cup_mF)^0$ such that $V\subseteq (E\cup_mF)^0_{\rm reg}$, $s_\cup^{-1}(V)$ is compact, and $s_\cup(s_\cup^{-1}(V))=V$. 
Consider the neighborhood $V\cap E^0$  of $v$ in $E^0$. It is clear that $V\cap E^0$ and $s_E^{-1}(V\cap E^0)$ are compact as closed subsets of compact sets $V$ and $s^{-1}_\cup(V)$, respectively. 
Hence, to prove that $v\in E^0_{\rm reg}$, it suffices to prove that $V\cap E^0\subseteq s_E(s_E^{-1}(V\cap E^0))$. 

From $s_\cup(s_\cup^{-1}(V))=V$, we infer that for every $w\in V\cap E^0$ there is $e\in (E\cup_mF)^1$ such that $s_\cup(e)=w$. If $e\in E^1$, then $w=s_\cup(e)=s_E(e)$. 
If $e\in F^1\setminus G^1$ then $w=s_\cup(e)=p^0(s_F(e))$, which implies that $s_F(e)\in G^0$. Since we assume that $p$ satisfies (F3), we get that $s_F(e)\in F^0_{\rm reg}$.
By negative invariance of $G^0$, there exist $x\in F^1$ such that $s_F(x)=s_F(e)$ and $r_F(x)\in G^1$. Hence, $x\in G^1$ due to positive invariance of~$G^0$, and we obtain that
$w=p^0(s_F(e))=p^0(s_F(x))=m^0(s_G(x))=s_E(m^1(x))$.
\end{proof}

Finally, we provide sufficient conditions leading to a positive answer to the question posed right before the statement of Proposition~\ref{pprop}.
First, we consider unions of topological graphs.
\begin{proposition}\label{adjfinal1}
Let $E\cup F$ be a union graph in which $E\cap F$ is a regular closed subgraph of both $E$ and $F$. Then $E$ and $F$ are regular closed subgraphs of $E\cup F$.
\end{proposition}
\begin{proof}
By Propositions~\ref{pprop} and~\ref{adjtb}, we can view $E$ and $F$ as closed subgraphs of $E\cup F$ satisfying $r_\cup^{-1}(E^0)\subseteq E^1$ and $r_\cup^{-1}(F^0)\subseteq F^1$, respectively. By Proposition~\ref{adjequi}, it suffices to prove that $F^0$ is negatively invariant in $(E\cup F)^0$. Since $F^0$ is positively invariant, its negative invariance is equivalent to the condition $F^0_{\rm sink}\cap (E\cup F)^0_{\rm reg}=\emptyset$ by~\cite[Proposition~2.2]{katsura3}. Suppose that $w\in F^0_{\rm sink}\cap (E\cup F)^0_{\rm reg}$. Then, there exists $e\in (E\cup F)^1$ such that $s_\cup(e)=w$. If $e\in F^1$, then we get a contradiction, so we suppose that $e\in E^1\setminus G^1$. Then $s_\cup(e)=s_E(e)=w$ and we conclude that $w\in G^0$. Regularity of $w$ in $E\cup F$ implies that $w\in E^0_{\rm fin}$. In turn, we get that $w\in G^0_{\rm fin}$. If $w\notin G^0_{\rm sink}$ then there exists $a\in G^1$ such that $s_F(a)=s_G(a)=w$, leading to another contradiction. Hence, we assume that $s_E(e)=w\in G^0_{\rm sink}\cap E^0_{\rm fin}\cap F^0_{\rm sink}$. If $w$ is regular in $E$, then invariance of $G$ in $E$ contradicts the assumption that $w\in F^0_{\rm sink}$. We infer that $s_E(e)=w\in G^0_{\rm sink}\cap (\overline{E^0_{\rm sink}}\setminus E^0_{\rm sink})\cap F^0_{\rm sink}$. Let $W_F$ be a compact neighborhood of $w$ in $F^0$ such that $s_F^{-1}(W_F)=\emptyset$ and let $V\subseteq (E\cup F)^0_{\rm reg}$ be a compact neighborhood of $w$ in $(E\cup F)^0$ such that $s_\cup^{-1}(V)$ is compact and $s_\cup(s_\cup^{-1}(V))=V$. Since $w\in \overline{F^0_{\rm sink}}\setminus F^0_{\rm sink}$ and $W_E\cap V\cap F^0$ is a compact neighborhood of $w$ both in $E^0$ and $F^0$, we obtain that there exists a vertex $v\in V\cap E^0_{\rm sink}\cap F^0_{\rm sink}$, i.e. a regular vertex in $E\cup F$ which is a sink both in $E$ and $F$. We thus arrive at a contradiction and conclude that $F^0_{\rm sink}\cap (E\cup F)^0_{\rm reg}=\emptyset$.
\end{proof}
Proposition~\ref{adjfinal1} motivates the following definition, which we will use in the sequel.
\begin{definition}
A union $E\cup F$ of topological graphs $E$ and $F$ is called {\em regular} if $E\cap F$ is a regular closed subgraph of both $E$ and $F$.
\end{definition}

Next, we come back to general adjunction graphs.
\begin{proposition}\label{adjfinal2}
Let $E\cup_mF$ be an adjunction graph such that $G$ is a closed regular subgraph of $F$ and $m:G\to E$ is a regular factor map. If we have that
\begin{equation}\label{boundcond}
G^0_{\rm sink}\cap \left(\overline{F^0_{\rm sink}}\setminus F^0_{\rm sink}\right)\cap F^0_{\rm fin}=\emptyset,
\end{equation}
then $E$ is a regular closed subgraph of $E\cup_mF$ and $p:F\to E\cup_mF$ is a regular factor map.
\end{proposition}
\begin{proof}
From Propositions~\ref{pprop} and~\ref{adjtb}, we conclude that $E$ is a closed subgraph satisfying the condition $r_\cup^{-1}(E^0)\subseteq E^1$ and $p$ consists of proper maps satisfying the conditions (F1) and (F2). 
Since $E^0$ is positively invariant in $(E\cup_mF)^0$, its negative invariance is equivalent to the condition $E^0_{\rm sink}\cap (E\cup_mF)^0_{\rm reg}=\emptyset$ by~\cite[Proposition~2.2]{katsura3}. 
We prove that that the condition~\eqref{boundcond} implies negative invariance of $E^0$ in $(E\cup_mF)^0$. Then, the desired result will follow from Proposition~\ref{adjequi}. 

Suppose that $v\in E^0_{\rm sink}\cap (E\cup_mF)^0_{\rm reg}$. Then there exist $e\in (E\cup_mF)^1$ such that $s_\cup(e)=v$. If $e\in E^1$, we get a contradiction. If $e\in F^1\setminus G^1$, then $s_\cup(v)=p^0(s_F(e))$. This means that $s_F(e)\in G^0$. If $s_F(e)$ is regular in $F$ or $s_F(e)\notin G^0_{\rm sink}$, we get a contradiction. Hence, we suppose that $s_F(e)\in G^0_{\rm sink}\cap F^0_{\rm sing}$. It follows that $s_F(e)\notin F^0_{\rm sink}$. Next, by properness of $p^1$ and $p^0$ and regularity of $v$ in $E\cup_mF$, we conclude that $s_F(e)\in F^0_{\rm fin}$. Therefore, $s_F(e)\in G^0_{\rm sink}\cap (\overline{F^0_{\rm sink}}\setminus F^0_{\rm sink})\cap F^0_{\rm fin}$, which contradicts the condition~\eqref{boundcond}.
\end{proof}

Clearly, every adjunction graph $E\cup_m F$ with non-injective $m$ and discrete $F$ is regular. This means that there are many examples of adjunction graphs satisfying the condition~\eqref{boundcond} which are not unions. Next, we give an example of a regular union graph violating the condition~\eqref{boundcond} of Proposition~\ref{adjfinal2}.
\begin{example}
Let $F^0=[0,2]$ and let $F^1=(0,1)\cup\{2\}\subseteq [0,2]$. Next, let $r_F$ be the inclusion of $F^1$ in $F^0$ and let
\[
s_F(x):=\begin{cases}\frac{x}{2} & x\in (0,1)\\1 & x=2\end{cases}.
\]
Note that $s_F(F^1)$ is not preopen, so $F$ is not a row-finite graph. Put $G^0=\{1\}$ and $G^1=\emptyset$. It is straightforward to check that $G$ is a regular closed subgraph of $F$ and that
\[
G^0_{\rm sink}\cap \left(\overline{F^0_{\rm sink}}\setminus F^0_{\rm sink}\right)\cap F^0_{\rm fin}=\{1\}\neq \emptyset,
\]
violating the condition~\eqref{boundcond} of Proposition~\ref{adjfinal2}. Next, let $E^0=\{1,3\}$ and $E^1=\emptyset$. Then, clearly $G$ is a regular closed subgraph of $E$. Hence, $E\cup F$ is a regular union.
\end{example}

We end this section by identifying another class of adjunction graphs satisfying the condition~\eqref{boundcond} of Proposition~\ref{adjfinal2}. Let us now show that row-finitness implies the condition~\eqref{boundcond} of Proposition~\ref{adjfinal2}. In fact, a weaker condition suffices.
\begin{proposition}\label{rowimply}
Let $F$ be a topological graph in which $s_F(F^1)$ is a preopen subset of $F^0$. Then $(\overline{F^0_{\rm sink}}\setminus F^0_{\rm sink})\cap F^0_{\rm fin}=\emptyset$.
\end{proposition}
\begin{proof}
Suppose that $(\overline{F^0_{\rm sink}}\setminus F^0_{\rm sink})\cap F^0_{\rm fin}$ is not empty. From~\cite[Lemma~1.22]{katsura1}, we conclude that
\[
\left(\overline{F^0_{\rm sink}}\setminus F^0_{\rm sink}\right)\cap F^0_{\rm fin}\subseteq F^0_{\rm fin}\setminus F^0_{\rm sink}\subseteq s_F(F^1).
\]
Since $s_F(F^1)$ is preopen, we infer that $s_F(F^1)\subseteq F^0_{\rm reg}$, so we get a contradiction.
\end{proof}


\section{Pullbacks of topological graph C*-algebras}

The aim of this section is to prove that under suitable conditions the topological graph C*-algebra of a regular adjunction graph is a pullback of the C*-algebras of the topological graphs being glued. To this end, we study the admissible pairs coming from kernels of $*$-homomorphisms induced by factor maps. Then we discuss how our result generalizes pushout-to-pullback theorems proved in the context of directed graphs.

\subsection{Main results}

First, we study kernels of $*$-homomorphisms induced by regular factor maps.
\begin{lemma}\label{factker}
Let $G$ be a regular closed subgraph of a topological graph $E$ and let $\mu:C^*(E)\to C^*(G)$ be the $*$-homo\-mor\-phism induced by the inclusion $G\subseteq E$ and given by Proposition~\ref{functc*}. Then  $\ker\mu=I_\rho$ for the admissible pair $\rho=(G^0,G^0_{\rm sing})$.
	\end{lemma}
	\begin{proof}
	Since $\mu$ is gauge-equivariant, $\ker\mu$ is a gauge-invariant ideal. Therefore, $\ker\mu$ has to be of the form $I_\rho$ for some admissible pair $\rho=(Y^0,Z^0)$ by Theorem~\ref{idealclass}. 
First, recall that
	\[
	t^0_E(C_0(E^0\setminus Y^0))=\ker\mu\cap t^0_E(C_0(E^0)).
	\]
For any $f\in C_0(E^0)$, $t^0_E(f)\in\ker\mu$ if and only if $f(v)=0$ for all $v\in G^0$ which implies that $Y^0=G^0$. Consider the topological graph $E_\rho$ defined by~\eqref{katspush}, where $\rho=(G^0,Z^0)$. There is a~regular factor map $n:E_\rho\to E$ given by the inclusions $G\subseteq E$, $W_\rho^0\subseteq E$, and $r_{G}^{-1}(W_\rho)\subseteq E^1$, which induces a~surjective $*$-homomorphism $\nu:C^*(E)\to C^*(E_\rho)$ whose kernel is $I_\rho=\ker\mu$ (see~\cite[Remark~3.14]{katsura3}). Next, one can prove that $G$ is a regular closed subgraph of $E_\rho$ (see~\cite[Remark~3.24]{katsura2}). Hence, we obtain an injective regular factor map $k:G\to E_\rho$ inducing a $*$-homomorphism $\kappa:C^*(E_\rho)\to C^*(G)$. From the definitions of the maps involved, we get a~commutative diagram of $*$-homomorphisms:
\[
\begin{tikzcd}
	{C^*(E)} \arrow[rd,"\mu"] \arrow[d,"\nu"']& \\
    C^*(E_\rho) \arrow[r,"\kappa"] &  C^*(G).
\end{tikzcd}
\]
Indeed, it is clear that we have the equality $m=n\circ k$ of regular factor maps. Since $\ker\nu=\ker\mu$, we conclude that $\ker\kappa=\{0\}$. It follows from~\cite[Proposition~2.9]{katsura2} that $k$ is surjective which can only happen when $W_\rho=Z^0\cap G^0_{\rm reg}=\emptyset$. In turn, we infer that $Z^0=G^0_{\rm sing}$, which ends the proof. 
	\end{proof}

Now, let $E\cup_m F$ be a regular adjunction graph.  Then we have a commutative diagram
\begin{equation}\label{factdiag}
 \begin{tikzcd}[ampersand replacement=\&]
	\& {E\cup_mF} \\
	{E} \&\& {F} \\
	\& {G}
	\arrow[, right hook->, from=2-1, to=1-2]
	\arrow["m", <-, from=2-1, to=3-2]
	\arrow["p", <-, from=1-2, to=2-3]
	\arrow[, right hook->, from=3-2, to=2-3]
\end{tikzcd}
\end{equation}
inducing, by Proposition~\ref{functc*}, a commutative diagram of $\mathbb{T}$-equivariant $*$-homomorphisms
\begin{equation}\label{c*diag}
    \begin{tikzcd}[ampersand replacement=\&]
	\& {C^*(E\cup_mF)} \\
	{C^*(E)} \&\& {C^*(F).} \\
	\& {C^*(G)}
	\arrow["\pi_E"',->>, from=1-2, to=2-1]
	\arrow["\mu_E"', ->, from=2-1, to=3-2]
	\arrow["\pi_F", ->, from=1-2, to=2-3]
	\arrow["\mu_F", ->>, from=2-3, to=3-2]
\end{tikzcd}
\end{equation}

Before proceeding further, we recall the notion of a pullback from category theory. Let $\mathcal{C}$ be any category and let $f_X:X\to Z$ and $f_Y:Y\to Z$ be morphisms in $\mathcal{C}$. A {\em pullback} of $f_X$ and $f_Y$ is an object $P$ in $\mathcal{C}$ together with morphisms $p_X:P\to X$ and $p_Y:P\to Y$ making the diagram
\begin{equation}\label{pulldiag}
\begin{tikzcd}[ampersand replacement=\&]
	\& P \\
	X \& \& Y \\
	\& Z
	\arrow["p_X"',->, from=1-2, to=2-1]
	\arrow["f_X"', ->, from=2-1, to=3-2]
	\arrow["p_Y", ->, from=1-2, to=2-3]
	\arrow["f_Y", ->, from=2-3, to=3-2]
\end{tikzcd}
\end{equation}
commute. Moreover, the triple $(P,p_X,p_Y)$ satisfies the following universal property: if there is $D\in\mathcal{C}$ and morphisms $g_X:D\to X$ and $g_Y:D\to Y$ such that $f_X\circ g_X=f_Y\circ g_Y$, then there is a~unique morphism $u:P\to D$ such that $g_X\circ u=p_X$ and $g_Y\circ u=p_Y$. It follows that $P$ is unique up to an isomorphism. We say that the diagram~\eqref{pulldiag} is a {\em pullback diagram} in the category $\mathcal{C}$. 

It is well known that in the category of C*-algebras and $*$-homomorphisms, the pullback of $\mu_1:A\to C$ and $\mu_2:B\to C$ is
\[
A\oplus_CB:=\{(a,b)\in A\oplus B~:~\mu_1(a)=\mu_2(b)\}
\]
together with the projections onto $A$ and $B$. If the $*$-homomorphisms $\mu_1$ and $\mu_2$ are $\mathbb{T}$-equivariant, then the same C*-algebra with the natural diagonal action and the same projections defines the pullback in the category of $\mathbb{T}$-C*-algebras and $\mathbb{T}$-equivariant $*$-homomorphisms.

Now, we turn our attention to the case of regular unions of topological graphs (see Definition~\ref{adjgraph}) which are more tractable due to Proposition~\ref{adjfinal1}. In this case, we have the following commutative diagram
\begin{equation}\label{c*diaginj}
    \begin{tikzcd}[ampersand replacement=\&]
	\& {C^*(E\cup F)} \\
	{C^*(E)} \&\& {C^*(F).} \\
	\& {C^*(E\cap F)}
	\arrow["\pi_E"',->>, from=1-2, to=2-1]
	\arrow["\mu_E"', ->>, from=2-1, to=3-2]
	\arrow["\pi_F", ->>, from=1-2, to=2-3]
	\arrow["\mu_F", ->>, from=2-3, to=3-2]
\end{tikzcd}
\end{equation} 
We have the following result:
\begin{theorem} \label{mainthm}
Let $E\cup F$ be a regular union graph such that
\begin{equation}\label{maincond}
E^0_{\rm sing}\cap (E\cap F)^0_{\rm reg}\subseteq F^0_{\rm reg}.
\end{equation}
Then, the diagram~\eqref{c*diaginj} is a pullback diagram in the category of $\mathbb{T}$-C*-algebras and $\mathbb{T}$-$*$-ho\-mo\-mor\-phisms. In other words, there is a $\mathbb{T}$-equivariant $*$-isomorphism
\[
C^*(E\cup F)\cong C^*(E)\oplus_{C^*(E\cap F)} C^*(F).
\]
\end{theorem}

\noindent The following two lemmas are crucial for proving Theorem~\ref{mainthm}.

\begin{lemma}\label{capker}
Let $E\cup F$ be a regular union graph. Then
\[
\ker\pi_E\cap\ker\pi_F=\{0\},
\]
where $\pi_E:C^*(E\cup F)\to C^*(E)$ and $\pi_F:C^*(E\cup F)\to C^*(F)$ are $*$-homomorphisms in the diagram~\eqref{c*diaginj}.
\end{lemma}
\begin{proof}
First, note that $\ker\pi_E\cap\ker\pi_F$ is a~gauge-invariant ideal. Hence, $\ker\pi_E\cap\ker\pi_F=I_\rho$ for some admissible pair $\rho=(Y^0,Z^0)$ due to Theorem~\ref{idealclass}.
From Lemma~\ref{factker}, we conclude that 
$\ker\pi_E=I_{\rho_E}$ and $\ker\pi_F=I_{\rho_F}$ for the admissible pairs $\rho_E=(E^0,E^0_{\rm sing})$ and $\rho_F=(F^0,F^0_{\rm sing})$. Now, \cite[Proposition~2.13]{katsura3} implies that \[
(Y^0,Z^0)=\rho=\rho_E\cup\rho_F=(E^0\cup F^0, E^0_{\rm sing}\cup F^0_{\rm sing}).
\]
Combining~\cite[Proposition~2.10]{katsura3} and~\cite[Proposition~2.12]{katsura3}, we infer that
\[
(E\cup F)^0_{\rm sing}\subseteq E^0_{\rm sing}\cup F^0_{\rm sing}\subseteq (E\cup F)^0\cap (E\cup F)^0_{\rm sing}=(E\cup F)^0_{\rm sing}.
\]
Hence, $\rho=((E\cup F)^0,(E\cup F)^0_{\rm sing})$. It follows that $\ker\pi_E\cap\ker\pi_F=\{0\}$.
\end{proof}

\begin{lemma}\label{reglemma}
Let $E\cup F$ be a regular union graph such that 
\begin{equation}\label{breaking}
E^0_{\rm sing}\cap (E\cap F)^0_{\rm reg}\subseteq  F^0_{\rm reg}. 
\end{equation}
Then
	$\ker\mu_E\subseteq \pi_E(\ker \pi_F)$, where $\mu_E$, $\pi_E$, and $\pi_F$ are the $*$-homomorphisms in the diagram~\eqref{c*diaginj}.
	\end{lemma}
	\begin{proof}
	By Lemma~\ref{factker}, $\ker\mu_E$ and $\ker \pi_F$ are gauge-invariant ideals with the associated admissible pairs $((E\cap F)^0,(E\cap F)^0_{\rm sing})$ and $(F^0,F^0_{\rm sing})$, respectively. Since $\pi_E$ is surjective and gauge-equivariant, $\pi_E(\ker\pi_F)$ is a gauge-invariant ideal. Therefore, by the classification of gauge-invariant ideals (see~Theorem~\ref{idealclass}), it suffices to prove that
	\[
	(Y^0_{\pi_E(\ker \pi_F)},Z^0_{\pi_E(\ker \pi_F)})\subseteq ((E\cap F)^0,(E\cap F)^0_{\rm sing}),
	\]
where $Y^0_{\pi_E(\ker \pi_F)}$ and $Z^0_{\pi_E(\ker \pi_F)}$ are given by~\eqref{admideal}. 

First, take a vertex $v\in Y^0_{\pi_E(\ker\pi_F)}$. Since $Y^0_{\ker\mu}=(E\cap F)^0$ by Lemma~\ref{factker}, it suffices to show that $f(v)=0$ for all $f\in C_0(E^0)$ such that $t_E^0(f)\in\ker\mu_E$. To this end, take any $f\in C_0(E^0)$ satisfying $t_E^0(f)\in\ker\mu_E$. By Tietze extension theorem, there is $f'\in C_0((E\cup F)^0)$ such that $f'|_{E^0}=f$ and $f'|_{F^0\setminus (E\cap F)^0}=0$.
Let $w\in F^0$ and consider $f'(w)$. If $w\in F^0\setminus (E\cap F)^0$, then $f'(w)=0$. Furthermore, if $w\in (E\cap F)^0$, then $f'(w)=f(w)=0$ because $t_E^0(f)\in\ker\mu_E$. Therefore, $f'|_{F^0}=0$. Since $\pi_F(t_\cup^0(f'))=t_F^0(f'|_{F^0})$, we conclude that $t_\cup^0(f')\in\ker\pi_F$. In turn, we get that \[
t_E^0(f)=t_E^0(f'|_{E^0})=\pi_E(t_\cup^0(f'))\in\pi_E(\ker\pi_F).
\] 
Hence, $f(v)=0$ and we are done.

Next, we prove that 
	\[
	Z_{\pi_E(\ker \pi_F)}\subseteq (E\cap F)^0_{\rm sing}.
	\]
	By contradiction, suppose that there exists $v\in Z_{\pi_E(\ker \pi_F)}$ such that $v\notin (E\cap F)^0_{\rm sing}$. Since
	\[
	Z_{\pi_E(\ker \pi_F)}\subseteq E^0_{\rm sing}\cap Y_{\pi_E(\ker \pi_F)}\subseteq E^0_{\rm sing}\cap (E\cap F)^0,
	\]
	we conclude that 
	$v\in E^0_{\rm sing}\cap (E\cap F)^0_{\rm reg}$,
	which implies that $v\in F^0_{\rm reg}$ by the assumption~\eqref{breaking}. Consequently, $v\notin F^0_{\rm sing}$. 
Since $(E\cup F)^0$ is locally compact Hausdorff (and hence Tychonoff), $F^0_{\rm sing}$ is closed, and $v\notin F^0_{\rm sing}$, there exists a function $f\in C_0((E\cup F)^0)$ such that $f(w)=0$ for all $w\in F^0_{\rm sing}$ and $f(v)\neq 0$. Since $Z_{\ker \pi_F}=F^0_{\rm sing}$, we have that
	\[
	t^0_\cup(C_0((E\cup F)^0\setminus F^0_{\rm sing}))=(\ker \pi_F+t^{(1)}_\cup(\mathcal{K}(X_{E\cup F})))\cap t^0_\cup (C_0((E\cup F)^0)).
	\]
	Since $f\in C_0((E\cup F)^0\setminus F^0_{\rm sing})$, we infer that
	\[
	t^0_\cup(f)\in\ker \pi_F+t_\cup^{(1)}(\mathcal{K}(X_{E\cup F})).
	\]
	Now, notice that
	\[
	\pi_E (t_\cup^{(1)}(\mathcal{K}(X_{E\cup F})))=t_E^{(1)}(\pi_E^{(1)}(\mathcal{K}(X_{E\cup F})))\subseteq t_E^{(1)}(\mathcal{K}(X_E)).
	\]
	Subsequently, we obtain that
	\[
	t_E^0(f|_{E^0})=\pi_E(t^0_\cup(f))\in \pi_E(\ker \pi_F+t_\cup^{(1)}(\mathcal{K}(X_{E\cup_mF})))\subseteq \pi_E(\ker \pi_F)+t_E^{(1)}(\mathcal{K}(X_E)).
	\]
	However, $f(v)\neq 0$, which contradicts the fact that $v\in Z_{\pi_E(\ker\pi_F)}$.
	\end{proof}
	\begin{remark}\label{remarksing}
	Note that $E^0_{\rm sink}\cap (E\cap F)^0_{\rm reg}=\emptyset$ so it suffices to assume that \[
	\overline{E^0_{\rm sink}}\setminus E^0_{\rm sink}\cap E^0_{\rm inf}\cap (E\cap F)^0_{\rm reg}\subseteq F^0_{\rm reg}.\]
	\end{remark}

We are now ready to complete the proof of the main result.
\begin{proof}
(of Theorem~\ref{mainthm}). From~\cite[Proposition~3.1]{pedersen} it follows that the diagram~\eqref{c*diaginj} is a pullback if and only if
\begin{enumerate}
\item[(P1)] $\ker\pi_E\cap\ker\pi_F=\{0\}$,
\item[(P2)] $\mu_F^{-1}(\mu_E(C^*(E)))=\pi_F(C^*(E\cup F))$,
\item[(P3)] $\pi_E(\ker\pi_F)=\ker\mu_E$.
\end{enumerate}
The condition (P1) follows from Lemma~\ref{capker}. The condition (P2) is true due to the commutativity of the diagram~\eqref{c*diaginj} and the surjectivity of $\mu_E$ and $\pi_F$. Finally, the condition (P3) follows from the commutativity of the diagram and the condition~\eqref{maincond} combined with Lemma~\ref{reglemma}. 
\end{proof}

\begin{remark}
We will not pursue it here but one might prove that the C*-correspondence of a~regular union graph $E\cup F$ is always a pullback of the C*-correspondences of $E$ and $F$ over the C*-correspondence of $E\cap F$. Therefore, an alternative way of proving Theorem~\ref{mainthm} would be to use~\cite[Theorem~3.3]{rs-11} stating under which conditions two-surjective pullbacks of C*-correspondences give rise to two-surjective pullbacks of their Cuntz--Pimsner algebras. However, in the case of C*-correspondences coming from topological graphs, these conditions are rather restrictive, namely, one would have to assume that $E^0=E^0_{\rm fin}$, $F^0=F^0_{\rm fin}$, $\overline{s_E(E^1)}$ is clopen in $E^0$, and $\overline{s_F(F^1)}$ is clopen in $F^0$. These conditions ensure that the gauge-invariant ideals of $C^*(E)$ and $C^*(F)$ are in one-to-one correspondence with invariant subsets (cf.~\cite[Corollary~8.23]{mt-05}), i.e. the $Z^0$ component of admissible pairs $(Y^0,Z^0)$ is not needed to classify those ideals. Therefore, our approach gives a more general result for Cuntz--Pimsner algebras coming from topological graphs.
\end{remark}

We end this section, by comparing our main theorems to an analogous result proved by the third author and P.~M.~Hajac in~\cite[Theorem~6.5]{th-24} in the case of directed graphs. The aforementioned paper deals with one-injective pushouts of directed graphs, which are equivalent to adjunction graphs, and the theorem is formulated using breaking vertices. For reader's convenience, we recall this notion here. Let $E$ be a directed graph and let $H$ be a~subset of $E^0$. The set of {\em breaking vertices} with respect to $H$ is defined to be
\[
B_H:=\{v\in E^0\setminus H~:~|s^{-1}_E(v)|=\infty~\text{and}~0<|s_E^{-1}(v)\cap r_E^{-1}(E^0\setminus H)|<\infty\}.
\]
Note that if $F$ is a regular subgraph of $E$, then $B_{E^0\setminus F^0}=E^0_{\rm inf}\cap F^0_{\rm reg}$. Using our notation, \cite[Theorem~6.5]{th-24} reads: Let $E\cup_mF$ be a regular adjunction graph of (discrete) directed graphs such that
\begin{enumerate}
\item[{\rm (D1)}] $m^0$ restricted to $B_{F^0\setminus G^0}$ is injective, and
\item[{\rm (D2)}] $p^0(B_{F^0\setminus G^0})\subseteq B_{(E\cup_mF)^0\setminus E^0}$.
\end{enumerate}
Then, the diagram~\eqref{c*diag} exists and is a pullback diagram in the category of $\mathbb{T}$-C*-algebras and $\mathbb{T}$-$*$-homo\-mor\-phisms, i.e. we have a $\mathbb{T}$-equivariant $*$-isomorphism
\[
C^*(E\cup_mF)\cong C^*(E)\oplus_{C^*(G)}C^*(F).
\]
Note that, the following conditions are equivalent:
\begin{enumerate}
\item[{\rm (1)}] $p^0(B_{F^0\setminus G^0})\subseteq B_{(E\cup_mF)^0\setminus E^0}$,
\item[{\rm (2)}] $E^0_{\rm sing}\cap m(G)^0_{\rm reg}\subseteq p(F)^0_{\rm reg}$.
\end{enumerate}
Therefore, a natural next step is to find a pushout-to-pullback theorem generalizing Theorem~\ref{mainthm} to the case when $m$ might not be injective but assuming a set of conditions analogous to (D1) and (D2).

\section{Examples and applications}

\subsection{Homeomorphism C*-algebras}

A~{\em topological dynamical system} $(X,\sigma)$ consits of a locally compact Hausdorff topological space $X$ and a homeomorphism $\sigma : X \xrightarrow{} X$. The map $\sigma$ defines an automorphism $\alpha$ of $C_0(X)$ given by $\alpha(f)(x) := f(\sigma^{-1}(x))$. The resulting crossed product C*-algebra $C^*(X) \rtimes_{\alpha} \mathbb{Z}$ is called the {\em homoeomorphism C*-algebra}. By Example 2 of \cite{katsura1}, this C*-algebra can be realised as a topological graph algebra with the topological graph given by $E^0 = E^1:= X, \ s := \sigma , \ r :={\rm id}_X $. Note that, since $\sigma$ is a~homeomorphism, the topological graph $E$ is row-finite. In fact, $E^0_{sg} = \emptyset$ and $s(E^1) = E^0 = E^0_{rg}$. First, let use identify regular factor maps between topological graphs coming from topological dynamical systems.

\begin{proposition}
Let $(X,\sigma_X)$ and $(Y,\sigma_Y)$ be topological dynamical systems and let $\phi:X\to Y$ be a continuous map. Then $(\phi,\phi)$ is a regular factor map if and only if
\begin{enumerate}
\item[{\rm (1)}] $\phi$ is proper,
\item[{\rm (2)}] $\phi$ is injective, and
\item[{\rm (3)}] $\phi\circ\sigma_X=\sigma_Y\circ \phi$.
\end{enumerate} 
\end{proposition}
\begin{proof}
Regular factor maps are proper by definition which is equivalent to the condition (1). Next, the condition (F1) is equivalent to the condition (3). Furthermore, (F2) is tantamount to injectivity of $\phi$. Indeed, if $\phi$ is injective, then $(\phi,\phi)$ clearly satisfies the target-bijectivity condition. On the other hand, if $(\phi,\phi)$ is target bijective then, for every $y\in Y$, there is a unique $x\in X$ such that $\phi(x)={\rm id}_Y(y)=y$. Finally, all topological graphs in sight have only regular vertices, so the condition (F3) is automatic.
\end{proof}
\begin{remark}
Notice that in topological dynamics factor maps are assumed to be surjective (see, e.g.~\cite[p.~38]{dv-14}).
\end{remark}

Now, let $(X,\sigma_X)$ and $(Y,\sigma_Y)$ be two dynamical systems, and $m: G \subseteq Y \xrightarrow[]{} X$ be such that $(m,m)$ is a regular factor map. Let $X \cup Y$ be the corresponding union and let $\sigma_{\cup}$ be a~homeomorphism of $X \cup Y$ defined by $\sigma_X$ on $X$ and $\sigma_Y$ on $Y$, respectively. Since $\sigma_X$ and $\sigma_Y$ agree on $G$, $\sigma_\cup$ is well-defined. We have the following commutative diagram of topological spaces:
\[\begin{tikzcd}
	& {X\cup Y} \\
	{X} && {Y} \\
	& {G}
	\arrow[from=2-1, to=1-2]
	\arrow[from=2-3, to=1-2]
	\arrow[from=3-2, to=2-1]
	\arrow[from=3-2, to=2-3]
\end{tikzcd}.\]
Dually, we obtain a pullback diagram in the category of commutative $\mathbb{Z}$-C*-algebras and $\mathbb{Z}$-$*$-homomorphisms:
\[\begin{tikzcd}
	& {C_0(X\cup Y)} \\
	{C_0(X)} && {C_0(Y)} \\
	& {C_0(G)}
	\arrow[from=1-2, to=2-1]
	\arrow[from=1-2, to=2-3]
	\arrow[from=2-1, to=3-2]
	\arrow[from=2-3, to=3-2]
\end{tikzcd}\]
Clearly, the above union $X\cup Y$ defines a regular union topological graph $(X \cup Y, X \cup Y,\sigma_{\cup}, {\rm id}_{X\cup Y})$.
Therefore, due to Theorem~\ref{mainthm}, we have the following pullback of C*-algebras:
\[\begin{tikzcd}
	& {C_0(X\cup_mY)\rtimes_{\alpha_{\cup}}\mathbb{Z}} \\
	{C_0(X)\rtimes_{\alpha_{X}}\mathbb{Z}} && {C_0(Y)\rtimes_{\alpha_{Y}}\mathbb{Z}} \\
	& {C_0(G)\rtimes_{\alpha_{G}}\mathbb{Z}}
	\arrow[from=1-2, to=2-1]
	\arrow[from=1-2, to=2-3]
	\arrow[from=2-1, to=3-2]
	\arrow[from=2-3, to=3-2]
\end{tikzcd}\]
Note that our result recovers a special case of a well-known theorem connecting pullbacks of C*-algebras and the crossed product construction~\cite[Theorem~6.3]{pedersen}.

\subsection{Quantum double suspension, quantum spheres, and quantum balls}

The aim of this section is to prove that C*-algebras associated to odd-dimensional quantum spheres and quantum balls are topological graph C*-algebras. Next, using our main theorem, we will recover the quantum CW-complex structure of odd-dimensional quantum spheres. The case of odd-dimensional spheres was missing in~\cite{hrt-20} and~\cite{th-24} because these papers were limited to the theory of graph C*-algebras. First, we generalize the construction of a quantum double suspension to topological graphs. To this end, we recall the presentation of the Toeplitz algebra as a graph C*-algebra.

Let $\mathcal{T}$ denote the Toeplitz algebra, i.e. the universal C*-algebra generated by an isometry~$S$. Recall that $\mathcal{T}$ is isomorphic with the C*-algebra of the graph $E$ consisting of two vertices $\{v_1,v_2\}$ and two edges $\{e_1,e_2\}$ such that $s_E(e_1)=r_E(e_1)=v_1$,  $s_E(e_2)=v_1$, and $r_E(e_2)=v_2$. We denote the respective projections by $P_1$ and $P_2$ and partial isometries by $S_1$ and $S_2$. The aforementioned isomorphism $\mathcal{T}\to C^*(E)$ is given by the assignment $S\mapsto S_1+S_2$.

The {\em quantum double suspension} $\Sigma^2A$ of a unital C*-algebra $A$ was defined in~\cite[Definition~6.1]{hs-02} using a certain essential extension. Here we only need the following characterization which is equivalent to~\cite[Definition~4.1]{cs-11} (cf.~\cite[Proposition~2.1]{hs-08}): $\Sigma^2 A$ is isomorphic with the C*-subalgebra of $A\otimes\mathcal{T}$ generated by $\{a\otimes P_2~:~a\in A\}$, $1\otimes P_1$, $1\otimes S_1$, and $1\otimes S_2$. The C*-algebra $\mathcal{T}$ is nuclear so there is no need to specify the tensor product. If $A$ is a unital graph C*-algebra $C^*(E)$, then the quantum double suspension is isomorphic to the C*-algebra of a graph $\Sigma^2E$ constructed from $E$~(see~\cite[Proposition~6.2]{hs-02}). We will show that a similar result is true for topological graphs.
\begin{definition}
Let $E$ be a topological graph with a compact vertex space $E^0$. We define the {\em quantum double suspension graph} $\Sigma^2E$ as follows
\[
(\Sigma^2E)^0:=\{w\}\sqcup E^0,\qquad (\Sigma^2E)^1:=\{x\}\sqcup E^0\sqcup E^1,\qquad s_{\Sigma}(x):=w=:r_\Sigma(x),
\]
\[
s_\Sigma(v):=w,\quad r_\Sigma(v):=v, \quad v\in E^0,\qquad s_\Sigma(e):=s_E(e),\quad r_\Sigma(e):=r_E(e),\quad e\in E^1.
\]
\end{definition}
Notice that we chose $E^0$ to be compact to ensure that $C^*(E)$ is unital (see, e.g.~\cite[Proposition~7.1]{katsura2}).

\begin{proposition}
Let $E$ be a topological graph with compact $E^0$. Then there is an isomorphism of $\mathbb{T}$-C*-algebras $C^*(\Sigma^2E)\cong \Sigma^2C^*(E)$.
\end{proposition}
\begin{proof}
First, by definition, $A_{\Sigma^2E}\cong \delta_w\mathbb{C}\oplus C(E^0)$ as C*-algebras and $X_{\Sigma^2E}\cong \delta_x\mathbb{C}\oplus X_{E^0}\oplus X_E$ as right Hilbert $C(E^0)$-modules, where $X_{E^0}:=\{\xi|_{E^0}:\xi\in X_{\Sigma^2E}\}\cong C(E^0)$. Next, let $(t^0_E,t^1_E)$ be the universal covariant representation of $(X_E,A_E)$. We define the maps $\psi^0:A_{\Sigma^2E}\to \Sigma^2C^*(E)$ and $\psi^1:X_{\Sigma^2E}\to \Sigma^2C^*(E)$ by (the linear extension of) the formulas
\[
\psi^0(\delta_w,0):=1\otimes P_1,\qquad \psi^0(0,f):=t_E^0(f)\otimes P_2,
\]
\[
\psi^1(\delta_x,0,0):=1\otimes S_1,\qquad \psi^1(0,f,0):=t_E^0(f)\otimes S_2,\qquad \psi^1(0,0,\xi):=t_E^1(\xi)\otimes P_2.
\]
By our definition of $\Sigma^2C^*(E)$, the maps $\psi^0$ and $\psi^1$ are well defined. Since $P_1$ and $P_2$ are orthogonal and selfadjoint, $\psi^0$ is \mbox{a~$*$-homo}\-morphism. Also, $\psi^1$ is clearly linear. Furthermore, the fact that $(t^0_E,t^1_E)$ is a representation of $(X_E,A_E)$ on $C^*(E)$ combined with the relations between the generators of $\mathcal{T}$ implies that $(\psi^1,\psi^0)$ is a representation of $(X_{\Sigma^2E},A_{\Sigma^2E})$ on $\Sigma^2C^*(E)$. 

Note that $w$ is regular in $\Sigma^2E$ because $s_\Sigma^{-1}(w)=\{x\}\sqcup E^0$ is compact and nonempty. Now, observe that
\begin{align*}
\psi^0(\delta_w,0)&=1\otimes P_1=1\otimes S_1S_1^*+1\otimes S_2S_2^*=\psi^1(\delta_x,0)\psi^1(\delta_w,0)^*+\psi^1(0,1,0)\psi^1(0,1,0)^*\\
&=\psi^{(1)}(\theta_{(\delta_x,0),(\delta_x,0)}+\theta_{(0,1,0),(0,1,0)}),
\end{align*}
which implies that $\psi^0(\delta_w,0)\in\psi^{(1)}(\mathcal{K}(X_{\Sigma^2E}))$. It follows from~\cite[Proposition~2.11]{katsura1} that $\psi^0(\delta_w,0)=\psi^{(1)}(\varphi_{\Sigma^2E}(\delta_w,0))$. Let $f\in C_0(E^0_{\rm reg})$ and let $\varphi_E(f)\approx \sum_{i=1}^n\theta_{\xi_i,\eta_i}$. Then
\[
\psi^0(0,f)=t_E^{(1)}(\varphi_E(f))\otimes P_2
\approx \sum_{i=1}^n (t_E^1(\xi_i)\otimes P_2)(t_E^1(\eta_i)\otimes P_2)^*=\sum_{i=1}^n \psi^{(1)}(0,0,\xi_i) \psi^{(1)}(0,0,\eta_i)^*.
\]
Again, by~\cite[Proposition~2.11]{katsura1}, we conclude that $\psi^0(0,f)=\psi^{(1)}(\varphi_{\Sigma^2E}(0,f))$. Hence, the representation $(\psi^1,\psi^0)$ is covariant. 

By universality, we obtain a $*$-homomorphism $\psi:C^*(\Sigma^2E)\to \Sigma^2C^*(E)$. From the definition of $(\psi^1,\psi^0)$, we see that it is surjective. Finally, the  $\mathbb{T}$-action on $\Sigma^2C^*(E)\subseteq C^*(E)\otimes\mathcal{T}$ given by the restriction of the tensor product of the gauge actions on $C^*(E)$ and $\mathcal{T}$ makes $\psi$ a $\mathbb{T}$-equivariant $*$-homomorphism. Since $\psi^0$ is injective, the gauge-invariant uniqueness theorem~\cite[Theorem~4.5]{katsura1} implies that $\psi$ is injective.
\end{proof}

Given a regular factor map $m:G\to E$, we define $\Sigma^2m:\Sigma^2G\to \Sigma^2E$ as follows
\begin{align}
(\Sigma^2m)^0(w_G):=w_E,\quad (\Sigma^2m)^0(h)&:=m^0(h),\qquad h\in G^0,\\
(\Sigma^2m)^1(x_G):=x_E,\quad (\Sigma^2m)^1(h):=m^0(h),&\quad (\Sigma^2m)^1(g):=m^1(g),\qquad h\in G^0,~g\in G^1.\nonumber
\end{align}
Here $w_E$ and $w_F$ are the additional vertices in $\Sigma^2E$ and $\Sigma^2F$, respectively. Notice that $\Sigma^2m$ is a~regular factor map. In particular, if $G$ is a regular closed subgraph of $F$, then $\Sigma^2G$ is a regular closed subgraph of $\Sigma^2F$. A {\em graph isomorphism} of topological graphs is a pair of homeomorphisms between of vertex and edges spaces satisfying the condition (F1) from the definition of a regular factor map.
\begin{proposition}\label{itsus}
Let $E$ and $F$ be topological graphs with compact vertex and edge spaces, let $G$ be a regular closed subgraph of $F$, and let $m:G\to E$ be a regular factor map. Then we have the following isomorphism of topological graphs
\[
\Sigma^2E\cup_{\Sigma^2m}\Sigma^2F\cong\Sigma^2(E\cup_mF).
\]
\end{proposition}
\begin{proof}
Define the maps $c_E^0:(\Sigma^2E)^0\to \Sigma^2(E\cup_mF)^0$ and $c^1_E:(\Sigma^2E)^1\to (\Sigma^2(E\cup_mF))^1$ as follows:
\[
c_E^0(w_E):=w_\cup,\quad c^0_E(v):=v,\qquad v\in E^0,
\]
\[
c^1_E(x_E):=x_\cup,\quad c^1_E(v):=v,\quad c^1_E(e)=e,\qquad v\in E^0,~e\in E^1.
\]
Note that we view $E^0$ and $E^1$ as closed subsets of $E^0\cup_{m^0}F^0$ and $E^1\cup_{m^1}F^1$, respectively. Similarly, we define the maps $c_F^0:(\Sigma^2F)^0\to \Sigma^2(E\cup_mF)$ and $c^1_F:(\Sigma^2F)^1\to (\Sigma^2(E\cup_mF))^1$ by the formulas
\[
c^0_F(w_G)=c_F^0(w_F):=w_\cup,\quad c^0_F(q):=p^0_F(q),\qquad q\in F^0,
\]
\[
c^1_F(x_F):=x_\cup,\quad c^1_F(q):=p_F^0(q),\quad c^1_F(f)=p_F^1(f),\qquad q\in F^0,~f\in F^1.
\]
Note that $c^0_E((\Sigma^2m)^0(w_G))=w_\cup=c^0_F(w_G)$ and
\[ c^0_E((\Sigma^2m)^0(h))=c^0_E(m^0(h))=m^0(h)=p_F(h)=
c_F^0(h),\qquad h\in G^0.
\]
Much in the same way, one shows that $c^1_E\circ(\Sigma^2m)^1=c^1_F$. Therefore, we have two continuous maps 
\[
c^0:(\Sigma^2E\cup_{\Sigma^2m}\Sigma^2F)^0\longrightarrow (\Sigma^2(E\cup_mF))^0
\]
\[
c^1:(\Sigma^2E\cup_{\Sigma^2m}\Sigma^2F)^1\longrightarrow (\Sigma^2(E\cup_mF))^1,
\]
such that $c^i|_{(\Sigma^2E)^0}=c_E^i$ and $c^i\circ p_{\Sigma^2F}^i=c_F^i$, $i=0,1$. It is clear that $c^0$ and $c^1$ are bijections and that $c=(c^1,c^0)$ is a graph homomorphism. By compactness of all spaces in sight, we obtain that $c^0$ and $c^1$ are homeomorphisms and that $c$ is a graph isomorphism.
\end{proof}

\subsubsection{Odd-dimensional quantum spheres and balls}

The C*-algebras $C(B^n_q)$, $n\geq 0$, of continuous functions on noncommutative (quantum) balls were defined by Hong and Szyma\'nski in \cite{hs-02} and studied in~\cite{hs-08}. In the original definition, the parameter $q$ is not explicitly involved but one can show that these C*-algebras are $q$-deformations, $q\in[0,1]$, of the commutative C*-algebras $C(B^n)$ of continuous functions on $n$-dimensional balls (see~\cite{hs-08} for details). For $q\in[0,1)$, these algebras are isomorphic.

The C*-algebras $C(B_q^{2n})$, $n\geq 0$, of even-dimensional quantum balls are defined inductively as follows:
\begin{equation*}
    C(B_q^0) = \mathbb{C}, \hspace{1cm} C(B_q^{2n}) = \Sigma^2C(B_q^{2n-2}).
\end{equation*}
$C(B_q^{2n})$ are all graph algebras with the graphs given in~\cite[p.~611]{hs-08}. Similarly, the C*-algebras $C(B_q^{2n+1})$, $n\geq 0$, of odd-dimensional quantum balls are defined as follows:
\begin{equation*}
    C(B_q^1) = C([-1,1]), \hspace{1cm} C(B_q^{2n+1}) = \Sigma^2C(B_q^{2n-1}).
\end{equation*}
However, $C(B_q^{2n+1})$ are known not to be graph C*-algebras~\cite[p.~616]{hs-08}. Here, we will show that they are topological graph C*-algebras.

To begin with, all commutative C*-algebras $C(X)$, where $X$ is a compact Hausdorff space, can be realized as topological graph C*-algebras of a topological graph $E$ with the edge space $E^1=\emptyset$, the vertex space $E^0=X$, and the range and source maps given by the empty maps. Now, using the previous definition of the quantum double suspension and the quantum double suspension graph, $C(B_q^{2n+1})$ can be realized as topological graph C*-algebras whose topological graph is given by iterating the quantum double suspension procedure. For instance, a topological graph for $C(B_q^3)$ is $([-1,1]\sqcup \{v\},[-1,1] \sqcup \{ e \}, s, r) $, where $s(y) = v$ for all $y\in [-1,1]\sqcup \{e\}$, $r(x) = x$ for $x \in [-1,1]$, and $r(e) = v$. A neat way to visualize this is given below.
\begin{center}
\begin{figure}[h!]
    \centering
    \includegraphics[width=0.25\textwidth]{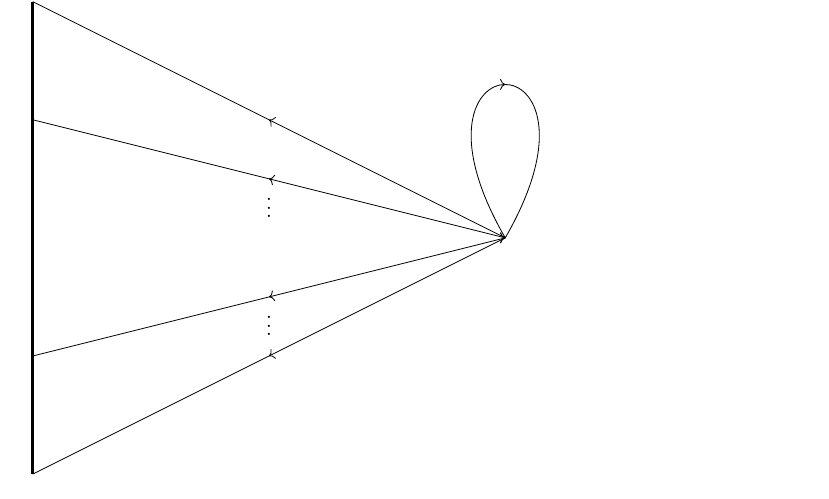}
    \captionsetup{labelformat=empty}
    \caption{Visualization of a topological graph for $C(B_q^3)$}
    \label{Bq3}
\end{figure}
\end{center}
Furthermore, the C*-algebras $C(S^{2n+1}_q)$ of odd-dimensional quantum spheres defined in \cite{hs-02} also can be realised inductively as:
\begin{equation*}
    C(S_q^0) = C(S^0) \hspace{1cm} C(S_q^1) = C(S^1) \hspace{1cm} C(S_q^n) = \Sigma^2C(S_q^{n-2}) 
\end{equation*}
Analogously to the case of $C(B_q^{3})$, a topological graph for $C(S_q^3)$ is shown below.
\begin{center} 
\begin{figure}[htp]
    \centering
    \includegraphics[width=0.25\textwidth]{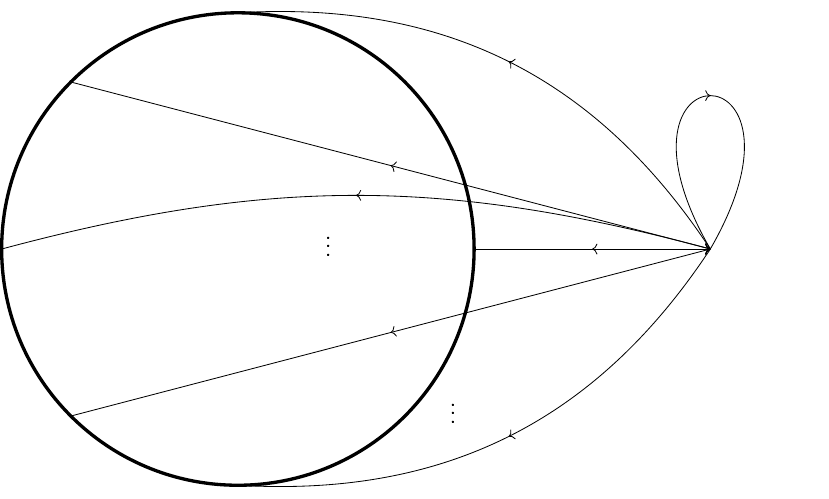}
    \captionsetup{labelformat=empty}
    \caption{Visualization of a topological graph for $C(S_q^3)$}
    \label{Sq3}
\end{figure}
\end{center}
Notice that the topological graph for the even-dimensional noncommutative balls and the even-dimensional quantum spheres agree with their corresponding directed-graph realizations.

With this characterization of the odd-dimensional noncommutative balls and quantum spheres we can construct the following pushout of the topological graphs:
    \[
    \begin{tikzcd}[ampersand replacement=\&]
	\& {\includegraphics[width=0.15\textwidth]{Bq3-2-pages-1.pdf}} \\
	{\includegraphics[width=0.15\textwidth]{Bq3-2-pages-2.pdf}} \&\& {\includegraphics[width=0.15\textwidth]{Bq3-2-pages-2.pdf}} \\
	\& {\includegraphics[width=0.15\textwidth]{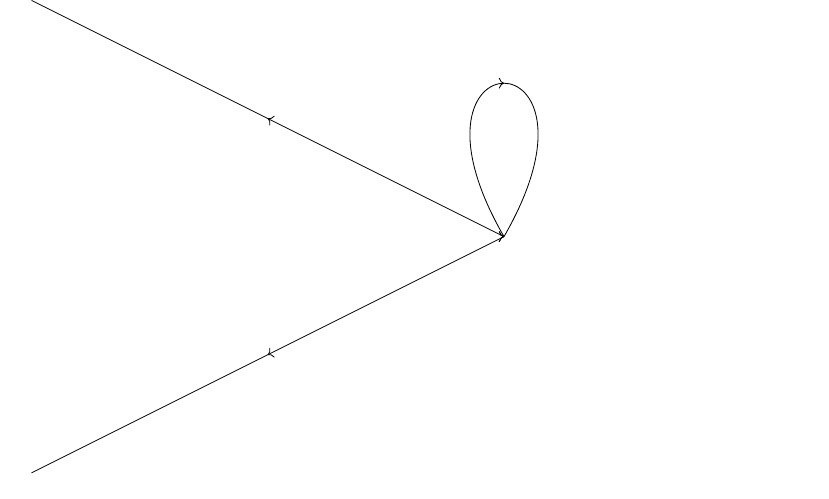}}
	\arrow[hook, from=2-1, to=1-2]
	\arrow[hook', from=2-3, to=1-2]
	\arrow[hook', from=3-2, to=2-1]
	\arrow[hook, from=3-2, to=2-3]
\end{tikzcd}.
\]
Here, all the inclusions and gluings are exactly as they are in the usual pushout construction of $S^1$ with additional vertices and edges remaining as they are. 

In general, one can iterate the above construction and use Proposition~\ref{itsus} to build an adjunction graph for $C(S^{2n+1}_q)$ for every $n\geq 0$. It is straightforward to check that this adjunction graph is regular and satisfies the condition~\ref{maincond} of Theorem~\ref{mainthm}. Therefore, we have the following pullback diagram in the category of $\mathbb{T}$-C*-algebras and $\mathbb{T}$-$*$-homomorphisms:
 \[
 \begin{tikzcd}[ampersand replacement=\&]
	\& {C(S_q^{2n+1})} \\
	{C(B_q^{2n+1})} \&\& {C(B_q^{2n+1})} \\
	\& {C(S_q^{2n+1})}
	\arrow[two heads, from=1-2, to=2-1]
	\arrow[two heads, from=1-2, to=2-3]
	\arrow[two heads, from=2-1, to=3-2]
	\arrow[two heads, from=2-3, to=3-2]
\end{tikzcd}.
\]


\section*{Acknowledgements} 
\noindent
This work is part of the project “Applications of graph algebras and higher-rank graph algebras
in noncommutative geometry” partially supported by NCN grant UMO-2021/41/B/ST1/03387.

The authors are grateful to Piotr M.~Hajac for numerous conversations regarding graph C*-algebras and pushout-to-pullback theorems. A.~Gothe and M.~Tobolski would like to thank the Arizona State University for the for excellent working conditions and great hospitality. A.~Gothe would also like to extend his gratitude to the University of Wroc\l{}aw for their kind hospitality.


\bibliographystyle{plain}
\bibliography{topgraph4}

\begin{thebibliography}{10}

\bibitem{bprs00}
Teresa Bates, David Pask, Iain Raeburn, and Wojciech Szyma\'{n}ski.
\newblock The {$C^*$}-algebras of row-finite graphs.
\newblock {\em New York J. Math.}, 6:307--324, 2000.

\bibitem{bs-24}
Samantha Brooker and Jack Spielberg.
\newblock Relative graphs and pullbacks of relative {T}oeplitz graph algebras.
\newblock {\em J. Operator Theory}, 92(1):25--47, 2024.

\bibitem{cs-11}
Partha~Sarathi Chakraborty and S.~Sundar.
\newblock Quantum double suspension and spectral triples.
\newblock {\em J. Funct. Anal.}, 260(9):2716--2741, 2011.

\bibitem{dhmsz-20}
Francesco D'Andrea, Piotr~M. Hajac, Tomasz Maszczyk, Albert Sheu, and Bartosz
  Zielinski.
\newblock The k-theory type of quantum cw-complexes, 2020.

\bibitem{dv-14}
Jan de~Vries.
\newblock {\em Topological dynamical systems}, volume~59 of {\em De Gruyter
  Studies in Mathematics}.
\newblock De Gruyter, Berlin, 2014.
\newblock An introduction to the dynamics of continuous mappings.

\bibitem{enchilada}
Siegfried Echterhoff, S.~Kaliszewski, John Quigg, and Iain Raeburn.
\newblock A categorical approach to imprimitivity theorems for
  {$C^*$}-dynamical systems.
\newblock {\em Mem. Amer. Math. Soc.}, 180(850):viii+169, 2006.

\bibitem{ek-79}
Hartmut Ehrig and Hans-J\"{o}rg Kreowski.
\newblock Pushout-properties: an analysis of gluing constructions for graphs.
\newblock {\em Math. Nachr.}, 91:135--149, 1979.

\bibitem{hrt-20}
Piotr~M. Hajac, Sarah Reznikoff, and Mariusz Tobolski.
\newblock Pullbacks of graph {$\rm C^*$}-algebras from admissible pushouts of
  graphs.
\newblock In {\em Quantum dynamics---dedicated to {P}rofessor {P}aul {B}aum},
  volume 120 of {\em Banach Center Publ.}, pages 169--178. Polish Acad. Sci.
  Inst. Math., Warsaw, 2020.

\bibitem{th-24}
Piotr~M. Hajac and Mariusz Tobolski.
\newblock From length-preserving pushouts of graphs to one-surjective pullbacks
  of graph algebras.
\newblock {\em J. Noncommut. Geom.}, 2025.
\newblock To appear.

\bibitem{hs-02}
Jeong~Hee Hong and Wojciech Szyma\'{n}ski.
\newblock Quantum spheres and projective spaces as graph algebras.
\newblock {\em Comm. Math. Phys.}, 232(1):157--188, 2002.

\bibitem{hs-08}
Jeong~Hee Hong and Wojciech Szyma\'{n}ski.
\newblock Noncommutative balls and mirror quantum spheres.
\newblock {\em J. Lond. Math. Soc. (2)}, 77(3):607--626, 2008.

\bibitem{katsura1}
Takeshi Katsura.
\newblock A class of {$C^\ast$}-algebras generalizing both graph algebras and
  homeomorphism {$C^\ast$}-algebras. {I}. {F}undamental results.
\newblock {\em Trans. Amer. Math. Soc.}, 356(11):4287--4322, 2004.

\bibitem{katsura0}
Takeshi Katsura.
\newblock On {$C^*$}-algebras associated with {$C^*$}-correspondences.
\newblock {\em J. Funct. Anal.}, 217(2):366--401, 2004.

\bibitem{katsura2}
Takeshi Katsura.
\newblock A class of {$C^*$}-algebras generalizing both graph algebras and
  homeomorphism {$C^*$}-algebras. {II}. {E}xamples.
\newblock {\em Internat. J. Math.}, 17(7):791--833, 2006.

\bibitem{katsura3}
Takeshi Katsura.
\newblock A class of {$C^*$}-algebras generalizing both graph algebras and
  homeomorphism {$C^*$}-algebras. {III}. {I}deal structures.
\newblock {\em Ergodic Theory Dynam. Systems}, 26(6):1805--1854, 2006.

\bibitem{kprr-97}
Alex Kumjian, David Pask, Iain Raeburn, and Jean Renault.
\newblock Graphs, groupoids, and {C}untz-{K}rieger algebras.
\newblock {\em J. Funct. Anal.}, 144(2):505--541, 1997.

\bibitem{kpsw-16}
Alex Kumjian, David Pask, Aidan Sims, and Michael~F. Whittaker.
\newblock Topological spaces associated to higher-rank graphs.
\newblock {\em J. Combin. Theory Ser. A}, 143:19--41, 2016.

\bibitem{mt-05}
Paul~S. Muhly and Mark Tomforde.
\newblock Topological quivers.
\newblock {\em Internat. J. Math.}, 16(7):693--755, 2005.

\bibitem{pedersen}
Gert~K. Pedersen.
\newblock Pullback and pushout constructions in {$C^*$}-algebra theory.
\newblock {\em J. Funct. Anal.}, 167(2):243--344, 1999.

\bibitem{pimsner}
Michael~V. Pimsner.
\newblock A class of {$C^*$}-algebras generalizing both {C}untz-{K}rieger
  algebras and crossed products by {${\bf Z}$}.
\newblock In {\em Free probability theory ({W}aterloo, {ON}, 1995)}, volume~12
  of {\em Fields Inst. Commun.}, pages 189--212. Amer. Math. Soc., Providence,
  RI, 1997.

\bibitem{rsz-04}
Iain Raeburn and Wojciech Szyma\'{n}ski.
\newblock Cuntz-{K}rieger algebras of infinite graphs and matrices.
\newblock {\em Trans. Amer. Math. Soc.}, 356(1):39--59, 2004.

\bibitem{rs-11}
David Robertson and Wojciech Szyma\'{n}ski.
\newblock {$C^\ast$}-algebras associated to {$C^\ast$}-correspondences and
  applications to mirror quantum spheres.
\newblock {\em Illinois J. Math.}, 55(3):845--870 (2013), 2011.

\bibitem{t-td08}
Tammo tom Dieck.
\newblock {\em Algebraic topology}.
\newblock EMS Textbooks in Mathematics. European Mathematical Society (EMS),
  Z\"{u}rich, 2008.

\end{thebibliography}

\Addresses

\end{document}